\documentclass[12pt]{amsart}

\usepackage{amssymb,amsmath,latexsym,enumerate,graphicx,bbm,mathptmx,microtype,
cite,mathtools}
\usepackage{hyperref,url}
\usepackage{multicol,enumitem}
\usepackage{verbatim,color}
\usepackage{subcaption}

\newtheorem{thm}{Theorem}
\numberwithin{thm}{section}
\newtheorem{cor}[thm]{Corollary}
\newtheorem{lem}[thm]{Lemma}
\newtheorem{prop}[thm]{Proposition}

\numberwithin{equation}{section}

\theoremstyle{definition}
\newtheorem{defin}[thm]{Definition}

\newtheorem{exa}[thm]{Example}

\usepackage{fullpage}

\renewcommand\emptyset{\varnothing}

\newcommand\commentout[1]{}

\newcommand\dist{\operatorname{dist}}
\newcommand\Em{\operatorname{Em}}

\newcommand\bI{\mathbf{I}}

\newcommand\Tstrut{\rule{0pt}{2.6ex}}         
\newcommand\Bstrut{\rule[-0.9ex]{0pt}{0pt}}   

\newcommand{\Inv}{\mathcal{I}}

\begin{document}


\baselineskip=17pt



\title{Consecutive Patterns in Inversion Sequences II: Avoiding Patterns of Relations}

\author{Juan S. Auli}
\address{Department of Mathematics\\
         Dartmouth College\\
         Hanover, NH 03755\\
         U.S.A.}
\email{juan.s.auli.gr@dartmouth.edu}

\author{Sergi Elizalde}
\address{Department of Mathematics\\
         Dartmouth College\\
         Hanover, NH 03755\\
         U.S.A.}
\email{sergi.elizalde@dartmouth.edu}
\urladdr{http://math.dartmouth.edu/~sergi}

\begin{abstract}
Inversion sequences are integer sequences $e=e_{1}e_{2}\dots e_{n}$ such that $0\leq e_{i}<i$ for each $i$.
The study of patterns in inversion sequences was initiated by Corteel--Martinez--Savage--Weselcouch and Mansour--Shattuck in the classical (non-consecutive) case, and later by Auli--Elizalde in the consecutive case, where the entries of a pattern are required to occur in adjacent positions.
In this paper we continue this investigation by considering {\em consecutive patterns of relations}, in analogy to the work of Martinez--Savage in the classical case. Specifically, given two binary relations $R_{1},R_2\in\{\leq,\geq,<,>,=,\neq\}$, we study inversion sequences $e$ with no subindex $i$ such that $e_{i}R_{1}e_{i+1}R_{2}e_{i+2}$.

By enumerating such inversion sequences according to their length, we obtain well-known quantities such as Catalan numbers, Fibonacci numbers and central polynomial numbers, relating inversion sequences to other combinatorial structures.
We also classify consecutive patterns of relations into Wilf equivalence classes, according to the number of inversion sequences avoiding them,
and into more restrictive classes that consider the positions of the occurrences of the patterns.

As a byproduct of our techniques, we obtain a simple bijective proof of a result of Baxter--Shattuck and Kasraoui about Wilf-equivalence of vincular patterns,
and we prove a conjecture of Martinez and Savage, as well as related enumeration formulas for inversion sequences satisfying certain unimodality conditions.
\end{abstract}

\keywords{Inversion sequence, pattern avoidance, pattern of relations, consecutive pattern, Wilf equivalence}

\subjclass[2010]{Primary 5A05; Secondary 05A15, 05A19.}


\maketitle

\section{Introduction}\label{sec:intro}

A common encoding of permutations is by their inversion sequences. Specifically, denoting by $S_n$ the set of permutations of $[n]=\{1,2,\dots,n\}$, and by $\bI_n$ the set of inversion sequences of length $n$ ---that is, integer sequences $e=e_{1}e_{2}\dots e_{n}$ with $0\leq e_{i}<i$ for each $i$---, one can define a bijection $\Theta:S_{n}\rightarrow\bI_{n}$ that assigns to each $\pi\in S_n$ its inversion sequence
\begin{equation}\label{eq:theta_bijection}
\Theta(\pi)=e=e_{1}e_{2}\dots e_{n}, \quad\textnormal{ where }\quad
e_{i}=\left|\{j:j<i\textnormal{ and }\pi_{j}>\pi_{i}\}\right|.
\end{equation}
Clearly, $e_1+\dots+e_n$ is the number of inversions of $\pi$, namely, pairs $(i,j)$ with $i<j$ and $\pi_i>\pi_j$.

In analogy to patterns in permutations, a research area that has received a lot of attention in the last few decades, one can study patterns in inversion sequences. In this context,
a {\em pattern} is a sequence $p=p_{1}p_{2}\dots p_{r}$ with $p_{i}\in\{0,1,\ldots,r-1\}$ for each $i$, where any value $j>0$ can appear in $p$ only if $j-1$ appears as well.
Given a word $w=w_1w_2\dots w_k$ over the integers, define its {\em reduction} to be the word obtained by
replacing all the occurrences of the $i$th smallest entry of $w$ with $i-1$ for all $i$. Then, an inversion sequence $e$ {\em contains} the
classical pattern $p=p_{1}p_{2}\dots p_{r}$ if there exists
a subsequence
$e_{i_{1}}e_{i_{2}}\dots e_{i_{r}}$ of $e$ (where $i_1<\dots<i_r$) with reduction $p$.
Otherwise, we say that $e$ {\em avoids} $p$. For instance, the inversion sequence $e=0014224$ avoids the pattern $210$, but it contains the pattern $101$ because $e_{4}e_{5}e_{7}=424$ has reduction $101$.

The study of classical patterns in inversion sequences was started by Corteel, Martinez, Savage and Weselcouch~\cite{MartinezSavageI}, and Mansour and Shattuck~\cite{MansourShattuck}. Their work connected classical patterns in inversion sequences to other combinatorial structures, which inspired a growing body of research on classical patterns in inversion sequences~\cite{Beaton, Bouvel, KimLin, KimLinII, Lin, MartinezSavageII, Yan}.

Motivated by~\cite{MartinezSavageI} and~\cite{MansourShattuck}, and by the growing interest in consecutive patterns in permutations~\cite{ElizaldeNoy,Elizalde}, we introduced consecutive patterns in inversion sequences and initiated an analogous study in~\cite{AuliElizalde}. In the definition below, the entries of a consecutive pattern are underlined to distinguish it from a classical pattern.

\begin{defin} An inversion sequence $e$ {\em contains} the consecutive pattern
$p=\underline{p_{1}p_{2}\dots p_{r}}$ if there is a consecutive subsequence
$e_{i}e_{i+1}\dots e_{i+r-1}$ of $e$ whose reduction is $p$. In this case, we call $e_{i}e_{i+1}\dots e_{i+r-1}$ an {\em occurrence} of $p$ in position $i$.
If $e$ does not contain $p$, then we say that $e$
{\em avoids} $p$. Denote by $\bI_n(p)$ the set of inversion sequences of length $n$ that avoid $p$.
\end{defin}

\begin{exa}\label{exa:contain_avoid_consec}
The inversion sequence $e=002241250\in\bI_{9}$ avoids the consecutive pattern $\underline{210}$, even though it contains the classical pattern $210$. On the other hand, $e$ contains $\underline{201}$ because $e_{5}e_{6}e_{7}$ is an occurrence of $\underline{201}$ in position 5.
\end{exa}

It is often useful to represent an inversion sequence $e$ as an underdiagonal lattice path from the origin to the line $x=n$, consisting of unit horizontal steps $E=(1,0)$ and unit vertical steps $N=(0,1)$ and $S=(0,-1)$. Each entry $e_{i}$ of $e$ is represented by a horizontal step: a segment between the points $(i-1,e_{i})$ and $(i,e_{i})$. Any necessary vertical steps are then inserted to make the path connected (see Figure~\ref{fig:path_represent}).

\begin{figure}[htb]
	\noindent \begin{centering}
	\includegraphics[scale=0.6]{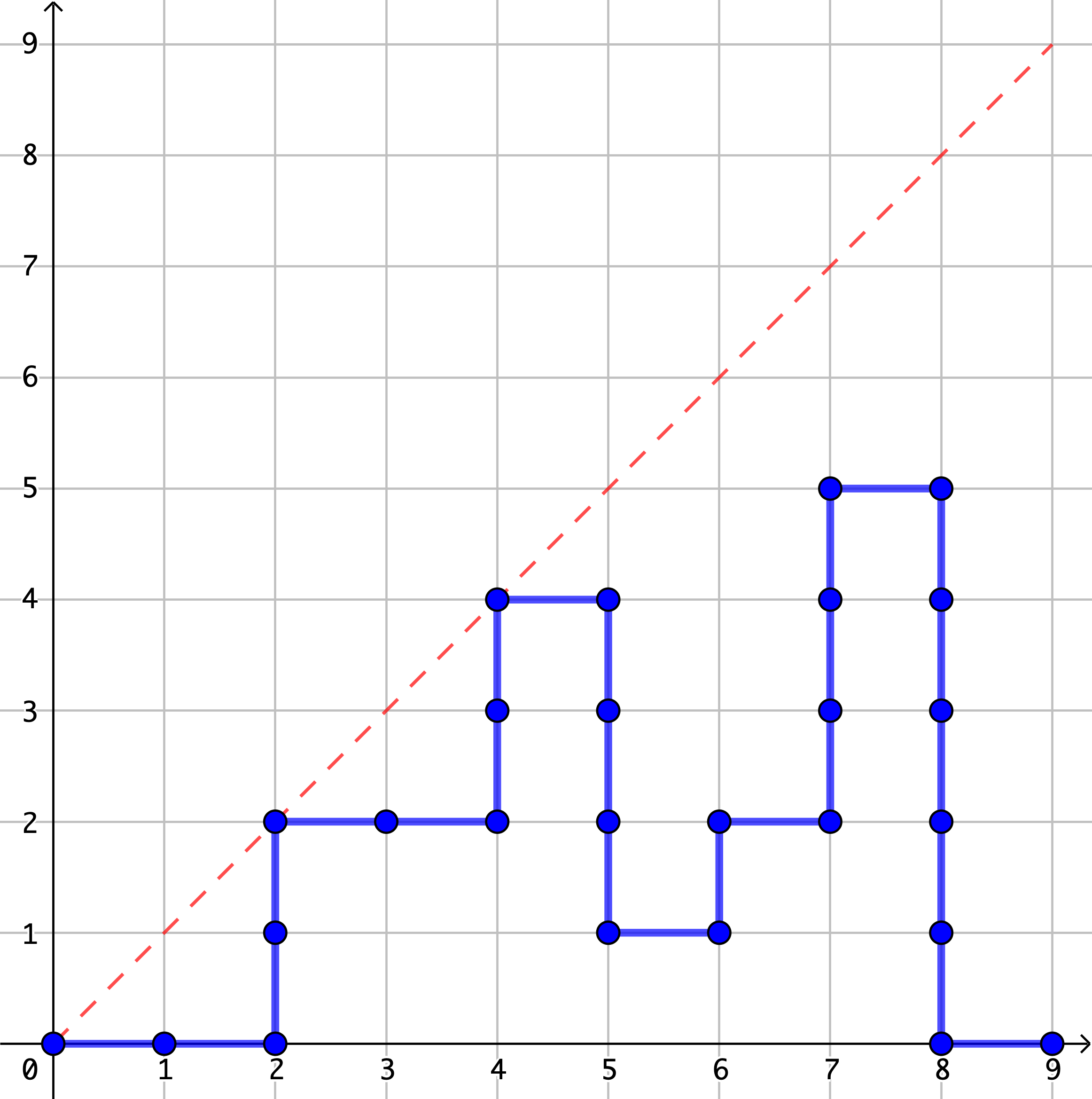}
		\par\end{centering}

	\protect\caption{Representation of $e=002241250\in\bI_{9}$ as a lattice path.\label{fig:path_represent}}
\end{figure}

Extending the systematic study of
Corteel {\it et al.}~\cite{MartinezSavageI} for classical patterns in inversion sequences,
Martinez and Savage~\cite{MartinezSavageII} reframe the notion of a
pattern of length 3 to instead consider a triple of binary relations between the entries of the occurrence. Given
a fixed triple of binary relations $\left(R_{1},R_{2},R_{3}\right)$, where
$R_{i}\in \{\leq,\geq,<,>,=,\neq,-\}$ for all $i$, they study the set
$\bI_{n}\left(R_{1},R_{2},R_{3}\right)$ consisting of those $e\in \bI_{n}$
with no subindices $i<j<k$ such that $e_{i}R_{1}e_{j}$, $e_{j}R_{2}e_{k}$ and
$e_{i}R_{3}e_{k}$. The symbol $-$ denotes the trivial relation where all elements are related, that is, $x-y$ for all $x,y$.

\begin{exa}
$\bI_{n}\left(\geq,\leq,\leq\right)$ is the set of inversion sequences
$e\in \bI_{n}$ with no $i<j<k$ such that $e_{i}\geq e_{j}$, $e_{j}\leq e_{k}$
and $e_{i}\leq e_{k}$; equivalently, the set of inversion sequences in $\bI_{n}$
avoiding all the patterns in the set $\{000,001,101,102\}$. On the other hand,
$\bI_{n}\left(\geq,\leq,-\right)$ denotes the set of inversion sequences
avoiding all the patterns in the set $\{000,001,100,101,102,201\}$.
\end{exa}

In this paper, we continue the work from~\cite{AuliElizalde} on consecutive patterns in inversion sequences by considering consecutive analogues of the notions introduced by Martinez and Savage~\cite{MartinezSavageII}. Specifically, we focus on the consecutive analogues of the sets $\bI_{n}\left(R_{1},R_{2},-\right)$. These are the most natural ones to study, since the general case would impose a restrictive relation $R_3$ between non-consecutive entries.

\begin{defin}
Let $R_{1},R_2\in\{\leq,\geq,<,>,=,\neq\}$. An inversion sequence $e$ {\em contains} the {\em consecutive pattern
of relations} $\left(\underline{R_{1},R_{2}}\right)$ if there is an $i$ such that $e_{i}R_{1}e_{i+1}$ and $e_{i+1}R_{2}e_{i+2}$.
In this case, we call $e_{i}e_{i+1}e_{i+2}$ an {\em occurrence} of $\left(\underline{R_{1},R_{2}}\right)$ in position $i$.
If $e$ does not contain $\left(\underline{R_{1},R_{2}}\right)$, then we say that $e$ {\em avoids} $\left(\underline{R_{1},R_{2}}\right)$. Denote by $\bI_{n}\left(\underline{R_{1},R_{2}}\right)$ the set of inversion sequences of length $n$ that avoid $\left(\underline{R_{1},R_{2}}\right)$.
\end{defin}

\begin{exa}
The inversion sequence $e=002241250$ contains $\left(\underline{>,<}\right)$ because $e_{5}e_{6}e_{7}=412$ is an occurrence of this pattern. However, $e$ avoids $\left(\underline{=,>}\right)$, and so $e\in\bI_{9}\left(\underline{=,>}\right)$.
\end{exa}

Note that an occurrence of a consecutive pattern $\left(\underline{R_{1},R_{2}}\right)$ in an inversion sequence is also an occurrence of some consecutive pattern of length 3. Thus, every set $\bI_{n}\left(\underline{R_{1},R_{2}}\right)$ can be expressed as an intersection $\bigcap_p\bI_{n}(p)$ where $p$ ranges over the consecutive patterns $p=\underline{p_1p_2p_3}$ satisfying $p_1R_1p_2$ and $p_2R_2p_3$. For instance, since occurrences of $\left(\underline{\geq,=}\right)$ are occurrences of either $\underline{100}$ or $\underline{000}$, we can write $\bI_{n}\left(\underline{\geq,=}\right)=\bI_{n}\left(\underline{000}\right)\cap\bI_{n}\left(\underline{100}\right)$.

This paper focuses on the enumeration of the sets $\bI_{n}\left(\underline{R_{1},R_{2}}\right)$.
These sets often exhibit more structure that the sets $\bI_{n}(p)$ avoiding a single consecutive pattern of length 3. Consequently, the sequences $\left|\bI_{n}\left(\underline{R_{1},R_{2}}\right)\right|$ are often simpler than the sequences $|\bI_{n}(p)|$ studied in~\cite{AuliElizalde},
and they provide more connections to other combinatorial objects and well-known integer sequences.

In addition to providing enumeration formulas, we also introduce several notions of equivalence between consecutive patterns of relations. These definitions are analogous to those for consecutive patterns in inversion sequences considered in~\cite{AuliElizalde}, which in turn are based on standard notions of equivalence between consecutive  patterns in permutations (see~\cite{DwyerElizalde,Elizalde}).

\begin{defin}\label{def:gen_Wilf_equiv} Let $\left(\underline{R_{1},R_{2}}\right)$ and $\left(\underline{R'_{1},R'_{2}}\right)$ be consecutive patterns of relations in inversion sequences. We say that $\left(\underline{R_{1},R_{2}}\right)$
and $\left(\underline{R'_{1},R'_{2}}\right)$ are
\begin{itemize}
\item {\it Wilf equivalent\/}, denoted by $\left(\underline{R_{1},R_{2}}\right)\sim \left(\underline{R'_{1},R'_{2}}\right)$, if
$\left|\bI_{n}\left(\underline{R_{1},R_{2}}\right)\right|=
\left|\bI_{n}\left(\underline{R'_{1},R'_{2}}\right)\right|$
for all~$n$;
\item {\it strongly Wilf equivalent\/}, denoted by
$\left(\underline{R_{1},R_{2}}\right)\stackrel{s}{\sim}\left(\underline{R'_{1},R'_{2}}\right)$,
if for each $n$ and $m$, the number of inversion sequences in
$\bI_{n}$ containing $m$ occurrences of $\left(\underline{R_{1},R_{2}}\right)$ is the
same as for $\left(\underline{R'_{1},R'_{2}}\right)$;
\item {\it super-strongly Wilf equivalent\/}, denoted by
$\left(\underline{R_{1},R_{2}}\right)\stackrel{ss}{\sim}\left(\underline{R'_{1},R'_{2}}\right)$,
if the above condition holds for any set of prescribed positions for the $m$ occurrences.
\end{itemize}
\end{defin}
We use the term {\em generalized Wilf equivalence} to refer to an equivalence of any one of the these three types.

\section{Summary of Results}

We will show that the 36 consecutive patterns of relations of the form $\left(\underline{R_{1},R_{2}}\right)$, with $R_{1},R_2\in{\{\leq,\geq,<,>,=,\neq\}}$ fall into
30 Wilf equivalence classes, and into 31 strong Wilf equivalence classes, which are also super-strong equivalence classes. The next result provides this classification. Patterns are listed from least avoided to most avoided in inversion sequences of length~$10$.
\begin{thm}\label{EquivIneq}  A complete list of the generalized
Wilf equivalences between consecutive patterns of relations $\left(\underline{R_{1},R_{2}}\right)$ in inversion sequences
is as follows:
\begin{multicols}{2}
 \begin{itemize}[itemsep=1ex,leftmargin=1.5cm]
 \item[(i)]$\left(\underline{\geq,<}\right)\stackrel{ss}{\sim}\left(\underline{<,\geq}\right)
 \sim\left(\underline{\neq,\geq}\right).$
 \item[(ii)] $\left(\underline{\geq,\geq}\right)\stackrel{ss}{\sim}\left(\underline{<,<}\right)$.
 \item[(iii)] $\left(\underline{\geq,=}\right)\stackrel{ss}{\sim}\left(\underline{=,\geq}\right)$.
 \item[(iv)] $\left(\underline{\geq,>}\right)
 \stackrel{ss}{\sim}\left(\underline{>,\geq}\right)$.
 \item[(v)] $\left(\underline{>,=}\right)\stackrel{ss}{\sim}\left(\underline{=,>}\right)$.
 \end{itemize}
\end{multicols}
\end{thm}

It is worth pointing out that the patterns $\left(\underline{\geq,<}\right)$ and $\left(\underline{\neq,\geq}\right)$ (similarly, $\left(\underline{<,\geq}\right)$ and $\left(\underline{\neq,\geq}\right)$) are Wilf equivalent but not strongly Wilf equivalent.

\begin{cor}\label{cor:Wilf-strongWilf}
Wilf equivalence and strong Wilf equivalence classes of consecutive patterns of relations in inversion sequences do not coincide in general.
\end{cor}

This result is interesting for two reasons. On the one hand, Wilf equivalence and strong Wilf equivalence classes of single consecutive patterns are conjectured to coincide, both in the setting of permutations (see Nakamura's conjecture \cite[Conjecture~5.6]{Nakamura}) and in the setting of inversion sequences (see~{\cite[Conjecture~2.3]{AuliElizalde}}). Corollary~\ref{cor:Wilf-strongWilf} shows that, perhaps surprisingly, the analogous statement for consecutive patterns of relations does not hold.
On the other hand, when considering consecutive patterns of relations in the setting of permutations, by
defining $\pi_i\pi_{i+1}\pi_{i+2}$  to be an occurrence of $\left(\underline{R_{1},R_{2}}\right)$ in $\pi\in S_n$ if $\pi_{i}R_{1}\pi_{i+1}$ and $\pi_{i+1}R_{2}\pi_{i+2}$,
Wilf equivalence and strong Wilf equivalence classes of patterns of the form $\left(\underline{R_{1},R_{2}}\right)$ in permutations coincide. In fact, all such equivalences are obtained from trivial symmetries, unlike in the case of consecutive patterns of relations in inversion sequences.

As a consequence of Theorem~\ref{EquivIneq}(iv), we will deduce the following result about permutation patterns, originally conjectured by
Baxter and Pudwell~{\cite[Conjecture 17]{BaxterPudwell}},
and later proved by
Baxter and Shattuck~{\cite[Corollary 11]{BaxterShattuck}} and by Kasraoui~\cite[Corolary~1.9(a)]{Kasraoui}. Here we present a direct bijective proof based on consecutive patterns of relations in inversion sequences, which is simpler than the previously known proofs. We write $\underline{124}3$ and $\underline{421}3$ to denote {\em vincular} (also called {\em generalized}) permutation patterns, where entries in underlined positions are required to be adjacent in an occurrence (see~\cite{Babson}).
We use $S_n(\sigma)$ to denote the set of permutations in $S_n$ that avoid a pattern~$\sigma$, and we say that two permutation patterns $\sigma$ and $\tau$ are Wilf equivalent if $|S_n(\sigma)|=|S_n(\tau)|$ for all~$n$.

\begin{cor}\label{BaxterandPudwell}
The vincular permutation patterns $\underline{124}3$ and $\underline{421}3$
are Wilf equivalent.
\end{cor}

Aside from the classification provided in Theorem~\ref{EquivIneq}, the other central result in this paper is the enumeration of inversion sequences avoiding consecutive patterns of relations.
We will show that, for many patterns $(\underline{R_1,R_2})$, the sequence $\left|\bI_{n}\left(\underline{R_{1},R_{2}}\right)\right|$ matches an existing sequence in the On-line Encyclopedia of Integer Sequences (OEIS)~\cite{OEIS} enumerating other combinatorial objects. In most cases, we will prove this bijectively.
These results are summarized in Table~\ref{tab3}. For other patterns of relations, even though we have no closed formulas for $\left|\bI_{n}\left(\underline{R_{1},R_{2}}\right)\right|$, we can obtain recurrences to compute these numbers.

\begin{table}[htbp]
\footnotesize
\begin{center}
 \begin{tabular}{ccp{4.1cm}l }
 \hline
 Pattern $\left(\underline{R_{1},R_{2}}\right)$ & OEIS~\cite{OEIS} & Description & Initial terms $\left(\underline{R_{1},R_{2}}\right)$ for $1\le n\le 9$
 \Tstrut\Bstrut\\
 \hline
 $\left(\underline{\leq,\neq}\right)$ & A040000 & $2$ \ (for $n>1$) & $1,2,2,2,2,2,2,2,2$ \Tstrut\\
 $\left(\underline{\leq,\geq}\right)$ & A000027 & $n$ & $1,2,3,4,5,6,7,8,9$ \Tstrut\\
  $\left(\underline{\geq,\neq}\right)$ & A000124 \Tstrut & $\binom{n}{2}+1$ & $1,2,4,7,11,16,22,29,37$\\
 $\left(\underline{\geq,\leq}\right)$ & A000045 & $F_{n+1}$ & $1,2,3,5,8,13,21,34,55$  \Tstrut\\
 $\left(\underline{\neq,\leq}\right)$ & A000071 \Tstrut & $F_{n+2}-1$  & $1,2,4,7,12,20,33,54,88$\\
 $\left(\underline{\geq,<}\right)\stackrel{ss}{\sim}\left(\underline{<,\geq}\right)
 \sim\left(\underline{\neq,\geq}\right)$ & A000079 & $2^{n-1}$ & $1,2,4,8, 16,32,64, 128, 256$\Tstrut\\
 $\left(\underline{\neq,\neq}\right)$ & A000085 & Number of involutions of $[n]$ & $1,2,4,10,26,76,232,764, 2620$\Tstrut\\
  $\left(\underline{\leq,>}\right)$ & A000108 & $C_{n}$ \ (Catalan) & $1,2,5,14,42,132,429,1430, 4862$\Tstrut\\
    $\left(\underline{>,\leq}\right)$ & A071356 \Tstrut & Underdiagonal paths of from the origin to $x=n$ with steps $(0,1)$, $(1,0)$, $(1,2)$ & $1,2,6,20,72,272,1064,4272, 17504$\\
 $\left(\underline{=,\neq}\right)$ & A003422 & $0!+1!+2!+\dots+(n-1)!$ & $1,2,4,10,34,154,874,5914, 46234$\Tstrut\\
 $\left(\underline{\geq,\geq}\right)\stackrel{ss}{\sim}\left(\underline{<,<}\right)$ & A049774 &
 $\left|S_{n}\left(\underline{321}\right)\right|$ & $1,2,5,17,70,349,2017,13358, 99377$\Tstrut\\
  $\left(\underline{\neq,=}\right)$ & A000522 \Tstrut & $\sum_{i=0}^{n-1}(n-1)!/i!$ & $1,2,5,16,65,326,1957,13700, 109601$\\
  $\left(\underline{\geq,>}\right)\stackrel{ss}{\sim}\left(\underline{>,\geq}\right)$ & A200403 &
  $\left|S_{n}\left(\underline{124}3\right)\right|$ & $1,2,6,23,107,584,3660,25910, 204564$\Tstrut\\
  $\left(\underline{=,=}\right)$ & A052169 & $\frac{\left(n+1\right)!-d_{n+1}}{n}$  & $1,2,5,19,91,531,3641,28673, 254871$\Tstrut\Bstrut\\
 \hline
 \end{tabular}
 \end{center}
 \caption{Consecutive patterns of relations $\left(\underline{R_{1},R_{2}}\right)$ for which
 $\left|\bI_{n}\left(\underline{R_{1},R_{2}}\right)\right|$ appears in~\cite{OEIS} and has an existing alternative combinatorial
 interpretation. Here $F_n$ denotes the $n$th Fibonacci number, and $d_n$ is the number of derrangements of $[n]$.
 The patterns are listed from least avoided to most avoided in inversion sequences of length~$10$.}\label{tab3}
\end{table}

As a byproduct of our enumeration, we will prove the following result involving non-consecutive triples of relations, which was
conjectured by Martinez and Savage~{\cite[Section 2.19]{MartinezSavageII}}. After presenting our solution at the conference {\it Permutation Patterns 2018}, we learned that it has also been proved independently using different methods by Cao, Jin and Lin~\cite[Theorem~5.1]{CaoJinLin} and by Hossain~\cite{Hossain}.

\begin{thm}\label{carlaConjecture}
The sequence $\left|\bI_{n}\left(>,\leq,-\right)\right|$ has ordinary
generating function (OGF)
\[
\sum_{n\ge0} \left|\bI_{n}\left(>,\leq,-\right)\right|\,z^n=\frac{1 + 2z - \sqrt{1 - 4z - 4z^{2}}}{4z}.
\]
\end{thm}

\medskip

The paper is organized as follows. Section~\ref{sec:gen_equiv} is devoted to the classification of consecutive patterns of relations into generalized Wilf equivalence classes, as given in Definition~\ref{def:gen_Wilf_equiv}, proving Theorem~\ref{EquivIneq}. We also provide a bijective proof of Corollary~\ref{BaxterandPudwell}. In Section~\ref{sec:enumerative_results}, we prove the enumerative results summarized in Table~\ref{tab3}, as well as Theorem~\ref{carlaConjecture}. Finally, in Section~\ref{sec:unimodal} we generalize the method used in the proof of Theorem~\ref{carlaConjecture} and apply it to enumerate different types of unimodal inversion sequences.

\section{Generalized Wilf Equivalences}\label{sec:gen_equiv}

\subsection{Proof of Theorem~\ref{EquivIneq}} In this section we will find all the generalized Wilf equivalences between consecutive patterns of relations. We start proving part~(v) of Theorem~\ref{EquivIneq}. Occurrences of $\left(\underline{>,=}\right)$ and $\left(\underline{=,>}\right)$ are simply occurrences of $\underline{100}$ and $\underline{110}$, respectively. Hence, the equivalence $\left(\underline{>,=}\right)\stackrel{ss}{\sim}\left(\underline{=,>}\right)$ is a restatement of the equivalence $\underline{100}\stackrel{ss}{\sim}\underline{110}$ (defined to mean that the number of inversion sequences of any given length with occurrences in any prescribed positions is the same for both patterns), which is proved in~\cite[Proposition 3.12]{AuliElizalde}. We repeat the proof here because the same method will be useful in proving several equivalences between consecutive patterns of relations.

First, we introduce some notation. Given $R_{1},R_2\in\{\leq,\geq,<,>,=,\neq\}$ and $e\in\bI_{n}$, define
\[
\Em\left(\left(\underline{R_{1},R_{2}}\right),e\right)=\left\{i:e_{i}e_{i+1}e_{i+2}\text{ is an occurrence of }\left(\underline{R_{1},R_{2}}\right)\right\}.
\]
Note that, by definition, $\left(\underline{R_{1},R_{2}}\right)\stackrel{ss}{\sim}\left(\underline{R'_{1},R'_{2}}\right)$ if and only if
\[
\left|\left\{e\in\bI_{n}:\Em\left(\left(\underline{R_{1},R_{2}}\right),e\right)= S\right\}\right|=\left|\left\{e\in\bI_{n}:\Em\left(\left(\underline{R'_{1},R'_{2}}\right),e\right)= S\right\}\right|
\]
for all $n$ and all $S\subseteq[n]$.
The next lemma, which is analogous to~\cite[Lemma 3.11]{AuliElizalde}, has a straightforward proof using the principle of inclusion-exclusion.

\begin{lem}\label{lem:inclusion_exlusion_consec_relations} Let $\left(\underline{R_{1},R_{2}}\right)$ and $\left(\underline{R'_{1},R'_{2}}\right)$ be two consecutive patterns of relations such that
\[
\left|\left\{e\in\bI_{n}:\Em\left(\left(\underline{R_{1},R_{2}}\right),e\right)\supseteq S\right\}\right|=\left|\left\{e\in\bI_{n}:\Em\left(\left(\underline{R'_{1},R'_{2}}\right),e\right)\supseteq S\right\}\right|
\]
for all positive integers $n$ and all $S\subseteq [n]$. Then $\left(\underline{R_{1},R_{2}}\right)\stackrel{ss}{\sim}\left(\underline{R'_{1},R'_{2}}\right)$.
\end{lem}

Now we are ready to prove~Theorem~\ref{EquivIneq}(v).

\begin{prop}\label{prop:100_and_110}
The patterns $\left(\underline{>,=}\right)$ and $\left(\underline{=,>}\right)$ are super-strongly Wilf equivalent.
\end{prop}
\begin{proof} By Lemma~\ref{lem:inclusion_exlusion_consec_relations}, it suffices to show that, for all positive integers $n$ and all $S\subseteq [n]$,
\[
\left|\left\{e\in\bI_{n}:\Em\left(\left(\underline{=,>}\right),e\right)\supseteq S\right\}\right|=\left|\left\{e\in\bI_{n}:\Em\left(\left(\underline{>,=}\right),e\right)\supseteq S\right\}\right|.
\]
To show this, we construct a bijection
\[\Phi_{S}:\left\{e'\in\bI_{n}:\Em\left(\left(\underline{=,>}\right),e'\right)\supseteq S\right\}\rightarrow\left\{e\in\bI_{n}:\Em\left(\left(\underline{>,=}\right),e\right)\supseteq S\right\}
\]
that replaces the occurrences of $\left(\underline{=,>}\right)$ (equivalently, of $\underline{110}$) in $e$ in positions $S$ with occurrences of $\left(\underline{>,=}\right)$ (equivalently, of $\underline{100}$). For $e\in\bI_{n}$ with $\Em\left(\left(\underline{=,>}\right),e'\right)\supseteq S$, define $\Phi_{S}(e)=e'$ by setting
\[
  e'_{j} = \begin{cases}
      e_{j+1} & \text{if } j-1\in S,\\
      e_{j} & \text{otherwise,}
      \end{cases}
\]
for $1\le j\le n$.  The sequence $e'$ is an inversion sequence because if $j-1\in S$, then $e'_j=e_{j+1}<e_j<j$.

To see that $\Phi_{S}$ is  a bijection, we describe its inverse $\Psi_{S}$. For $e'\in\bI_{n}$ with $\Em\left(\left(\underline{>,=}\right),e\right)\supseteq S$, define $\Psi_{S}(e')=e$ by
\[
  e_{j} = \begin{cases}
      e'_{j-1} & \text{if } j-1\in S,\\
      e'_{j} & \text{otherwise,}
      \end{cases}
\]
for $1\le j\le n$. Note that $e\in\bI_n$ because if $j-1\in S$, then $e_j=e'_{j-1}<j-1<j$.
Since no two occurrences of $\left(\underline{=,>}\right)$ (respectively, $\left(\underline{>,=}\right)$) can overlap in more than one entry, the maps $\Phi_{S}$ and $\Psi_{S}$ are inverses of each other.
\end{proof}

The next result proves Theorem~\ref{EquivIneq}(iii). It again relies on Lemma~\ref{lem:inclusion_exlusion_consec_relations}.

\begin{prop}\label{prop:sergi_equivalences_consec_relations_1} The patterns $\left(\underline{\geq,=}\right)$ and $\left(\underline{=,\geq}\right)$ are super-strongly Wilf equivalent.
\end{prop}

\begin{proof}
Let $n$ be a positive integer and $S\subseteq [n]$. By Lemma~\ref{lem:inclusion_exlusion_consec_relations}, it suffices to construct a bijection
\[
\Phi_{S}:\left\{e\in\bI_{n}:\Em\left(\left(\underline{=,\geq}\right),e\right)\supseteq S\right\}\rightarrow\left\{e'\in\bI_{n}:\Em\left(\left(\underline{\geq,=}\right),e'\right)\supseteq S\right\}.
\]
We can write $S$ uniquely as a disjoint union of {\em blocks}, which we define as maximal subsets whose entries are consecutive. Explicitly, write $S=\bigsqcup_{j=1}^{m}B_{j}$, where $B_{j}=\left\{i_{j},i_{j}+1,\ldots,i_{j}+l_{j}\right\}$, with $l_{j}\geq 0$ and $i_j+l_j+1<i_{j+1}$ for all $j$.

Let $e\in\bI_{n}$ be such that $\Em\left(\left(\underline{=,\geq}\right),e\right)\supseteq S$. Then, for each $1\leq j\leq m$,
we have $$e_{i_{j}}=e_{i_{j}+1}=\ldots =e_{i_{j}+l_{j}+1}\ge e_{i_{j}+l_{j}+2}.$$
Define $\Phi_S(e)=e'$ by setting
\[
  e'_{i} = \begin{cases}
      e_{i_{j}+l_{j}+2} & \text{if } i-1\in B_{j}\text{ for some }j\text{ and }e_{i_{j}+l_{j}+1}>e_{i_{j}+l_{j}+2},\\
      e_{i} & \text{otherwise,}
      \end{cases}
\]
for $1\le i\le n$, as described schematically in Figure~\ref{fig:Scheme_iii}.
In other words, we have
$$e_{i_{j}}=\ e'_{i_{j}}\ge e'_{i_{j}+1}=\ldots =e'_{i_{j}+l_{j}+1}= e'_{i_{j}+l_{j}+2}\ = e_{i_{j}+l_{j}+2}$$
for each $j$, and so $\Em\left(\left(\underline{\geq,=}\right),e'\right)\supseteq S$. Additionally, if $i$ is such that $e'_i\neq e_i$, then $e'_{i}=e_{i_{j}+l_{j}+2}\le e_{i}<i$. Hence $e'\in\bI_{n}$.

\begin{figure}[htb]
	\centering
	\includegraphics[scale=0.7]{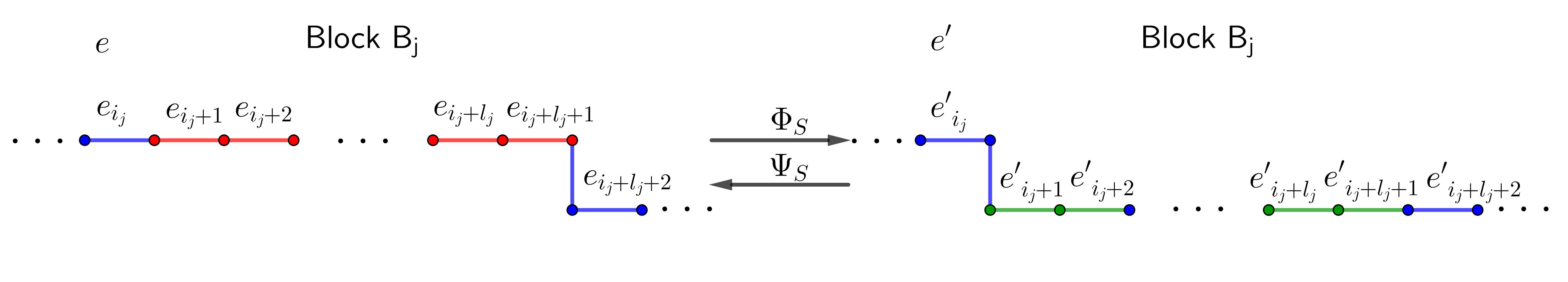}
\caption{Schematic diagram of the behavior of $\Phi_{S}$ and $\Psi_{S}$ from Proposition~\ref{prop:sergi_equivalences_consec_relations_1}.}\label{fig:Scheme_iii}
\end{figure}

To see that $\Phi_S$ is a bijection, we describe its inverse map $\Psi_S$ as follows. Given $e'\in\bI_{n}$ such that $\Em\left(\left(\underline{\geq,=}\right),e'\right)\supseteq S$, let $\Psi_{S}\left(e'\right)=e$, where
\[
  e_{i} = \begin{cases}
      e'_{i_{j}} & \text{if } i-1\in B_{j}\text{ for some }j\text{ and }e'_{i_{j}}>e'_{i_{j}+1},\\
      e'_{i} & \text{otherwise,}
      \end{cases}
\]
for $1\le i\le n$. If $i$ is such that $e_i\neq e'_i$, then $e_{i}=e'_{i_{j}}<i_{j}<i$, and so $e\in\bI_{n}$. Additionally, $\Em\left(\left(\underline{=,\geq}\right),e\right)\supseteq S$. By construction, the maps $\Phi_{S}$ and $\Psi_{S}$ are inverses of each other.
\end{proof}

Next we prove Theorem~\ref{EquivIneq}(iv), using a similar argument.

\begin{prop}\label{prop:sergi_equivalences_consec_relations_2} The patterns $\left(\underline{\geq,>}\right)$ and $\left(\underline{>,\geq}\right)$ are super-strongly Wilf equivalent.
\end{prop}
\begin{proof}
Let $n$ be a positive integer and $S\subseteq [n]$. By Lemma~\ref{lem:inclusion_exlusion_consec_relations}, it is enough to describe a bijection
\[
\Phi_{S}:\left\{e\in\bI_{n}:\Em\left(\left(\underline{\geq,>}\right),e\right)\supseteq S\right\}\rightarrow\left\{e'\in\bI_{n}:\Em\left(\left(\underline{>,\geq}\right),e'\right)\supseteq S\right\}.
\]
Again, we write $S$ as a disjoint union of blocks, $S=\bigsqcup_{j=1}^{m}B_{j}$, where $B_{j}=\left\{i_{j},i_{j}+1,\ldots,i_{j}+l_{j}\right\}$, with $l_{j}\geq 0$ and $i_j+l_j+1<i_{j+1}$ for all $j$.

Given $e\in\bI_{n}$ such that $\Em\left(\left(\underline{\geq,>}\right),e\right)\supseteq S$, we have
$$e_{i_{j}}\ge e_{i_{j}+1}>\ldots >e_{i_{j}+l_{j}+1}>e_{i_{j}+l_{j}+2}$$
for each $j$. Define $\Phi_S(e)=e'$ by setting
\[
  e'_{i} = \begin{cases}
      e_{i+1} & \text{if } i-1\in B_{j}\text{ for some }j\text{ and }e_{i_{j}}=e_{i_{j}+1},\\
      e_{i} & \text{otherwise,}
      \end{cases}
\]
for $1\le i\le n$, as shown in Figure~\ref{fig:Scheme_iv}. Then
$$e'_{i_{j}}> e'_{i_{j}+1}>\ldots >e'_{i_{j}+l_{j}+1}\ge e'_{i_{j}+l_{j}+2}$$
for each $j$, and so $\Em\left(\left(\underline{>,\geq}\right),e'\right)\supseteq S$.
Additionally, since $e'_{i}=e_{i+1}<e_i<i$ whenever $e'_i\neq e_i$, we have that $e'\in\bI_{n}$.

\begin{figure}[htb]
\centering
	\includegraphics[scale=0.7]{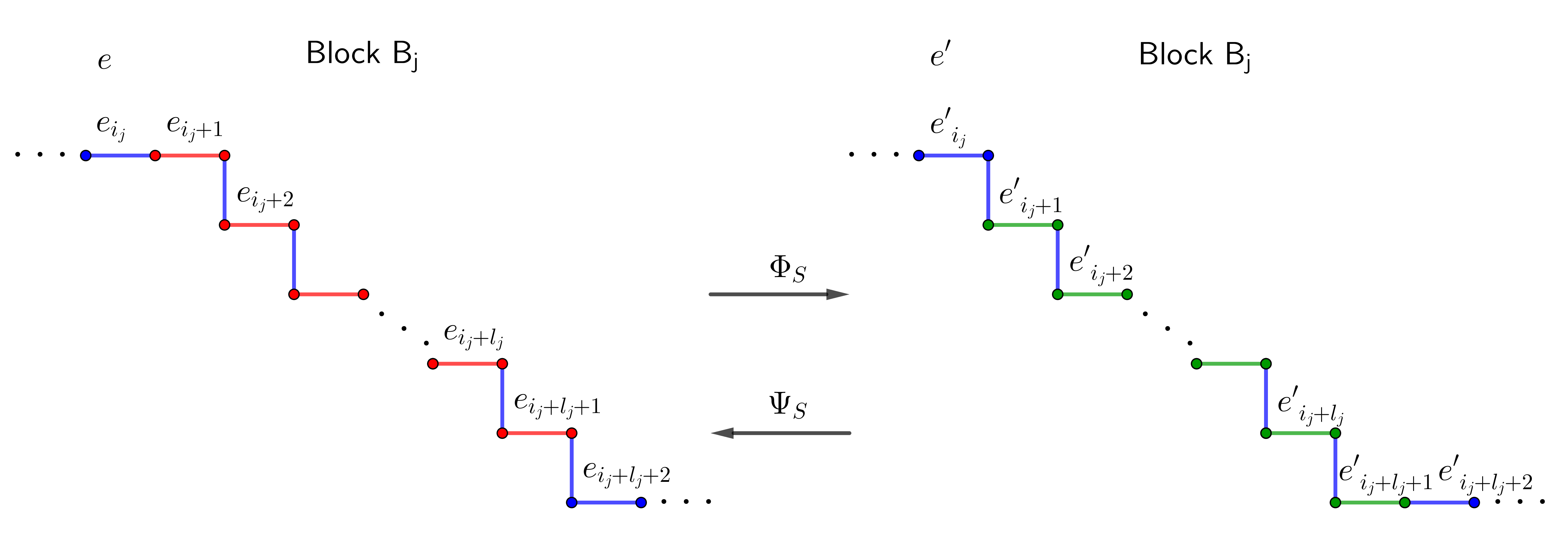}
\caption{Schematic diagram of the behavior of $\Phi_{S}$ and $\Psi_{S}$ from Proposition~\ref{prop:sergi_equivalences_consec_relations_2}.}\label{fig:Scheme_iv}
\end{figure}

To see that $\Phi_{S}$ is a bijection, note that its inverse $\Psi_{S}$ can be described as follows. Given $e'\in\bI_{n}$ such that $\Em\left(\left(\underline{>,\geq}\right),e'\right)\supseteq S$, let $\Psi_{S}\left(e'\right)=e$, where
\[
  e_{i} = \begin{cases}
      e'_{i-1} & \text{if } i-1\in B_{j}\text{ for some }j\text{ and }e'_{i_{j}+l_{j}+1}=e'_{i_{j}+l_{j}+2},\\
      e'_{i} & \text{otherwise.}
      \end{cases}
\]
This concludes the proof.
\end{proof}

Given $e=e_{1}e_{2}\ldots e_{n}\in\bI_{n}$, we define its {\em complement} to be the inversion sequence $e^C=e^{C}_{1}e^{C}_{2}\ldots e^{C}_{n}$, where $e^{C}_{i}=i-1-e_{i}$ for $1\leq i\leq n$.
This notion is useful in proving Theorem~\ref{EquivIneq}(i)(ii), which we do next.

\begin{prop}\label{prop:complements_equivalences_consec_relations} The following equivalences hold:
\begin{multicols}{2}
 \begin{itemize}[itemsep=1ex,leftmargin=1.5cm]
\item[(i$_1$)] $\left(\underline{\geq,<}\right)\stackrel{ss}{\sim}\left(\underline{<,\geq}\right)$.
\item[(ii)] $\left(\underline{\geq,\geq}\right)\stackrel{ss}{\sim}\left(\underline{<,<}\right)$.
\end{itemize}
\end{multicols}
\end{prop}

\begin{proof}
To prove part~(i$_1$), note that $e_{i}\geq e_{i+1}$ if and only if $e^{C}_{i}<e^{C}_{i+1}$. Indeed, the first equality is equivalent to
$i-1-e^C_{i}\ge i-e^C_{i+1}$, which is in turn equivalent to the second. We deduce that
$$e_{i}\geq e_{i+1}<e_{i+2}\quad\text{ if and only if }\quad e^{C}_{i}<e^{C}_{i+1}\geq e^{C}_{i+2}.$$

It follows that the involution $e\rightarrow e^{C}$ replaces occurrences of $\left(\underline{\geq,<}\right)$ by occurrences of $\left(\underline{<,\geq}\right)$, and vice versa.
We conclude that $\left(\underline{\geq,<}\right)$ and $\left(\underline{<,\geq}\right)$ are super-strongly Wilf equivalent.

A very similar argument proves part (ii).
\end{proof}

The next proposition proves the other equivalence in Theorem~\ref{EquivIneq}(i). Corollary~\ref{cor:Wilf-strongWilf} is an immediate consequence.

\begin{prop}\label{prop:wilf_equiv_equivalences_consec_relations} The patterns $\left(\underline{<,\geq}\right)$ and  $\left(\underline{\neq,\geq}\right)$ are Wilf equivalent, but not strongly Wilf equivalent.
\end{prop}

\begin{proof} To prove that $\left(\underline{<,\geq}\right)\sim\left(\underline{\neq,\geq}\right)$, we will show that, in fact, $\bI_{n}\left(\underline{<,\geq}\right)=\bI_{n}\left(\underline{\neq,\geq}\right)$.
Occurrences of $\left(\underline{<,\geq}\right)$ are occurrences of one of the following consecutive patterns: $\underline{010}$, $\underline{011}$, $\underline{021}$, $\underline{120}$. Similarly, occurrences of $\left(\underline{\neq,\geq}\right)$ are occurrences of one of the same four patterns, together with $\underline{100}$ or $\underline{210}$.
We deduce that $\bI_{n}\left(\underline{\neq,\geq}\right)\subseteq\bI_{n}\left(\underline{<,\geq}\right)$.

Suppose that $e\in\bI_{n}\left(\underline{<,\geq}\right)$ has an occurrence $e_{i}e_{i+1}e_{i+2}$ of $\left(\neq,\geq\right)$. Then it must be an occurrence of $\underline{100}$ or $\underline{210}$, so $e_{i}>e_{i+1}\ge 0$.
Let $j$ be the largest index such that $1\le j<i$ and $e_j<e_{j+1}$. Such a $j$ must exist because $e_{1}=0$ and $e_i>0$.
But then $e_j<e_{j+1}\ge e_{j+2}$, that is, $e_{j}e_{j+1}e_{j+2}$ is an occurrence of $\left(\underline{<,\geq}\right)$, which is a contradiction. We conclude that $\bI_{n}\left(\underline{<,\geq}\right)=\bI_{n}\left(\underline{\neq,\geq}\right)$.

To show that $\left(\underline{<,\geq}\right)$ and  $\left(\underline{\neq,\geq}\right)$ are not strongly Wilf equivalent, it is enough to look at the number of occurrences of these patterns in inversion sequences of length $4$. There is an inversion sequence in $\bI_{4}$ with two occurrences of $\left(\underline{\neq,\geq}\right)$, namely $0100$, but none with two occurrences of $\left(\underline{<,\geq}\right)$.
\end{proof}

\begin{proof}[Proof of Theorem~\ref{EquivIneq}]
Propositions \ref{prop:100_and_110}, \ref{prop:sergi_equivalences_consec_relations_1}, \ref{prop:sergi_equivalences_consec_relations_2}, \ref{prop:complements_equivalences_consec_relations} and \ref{prop:wilf_equiv_equivalences_consec_relations} prove all the equivalences listed. A brute force computation of the first 10 terms of the sequences $\left|\bI_{n}\left(\underline{R_{1},R_{2}}\right)\right|$ for $R_{1},R_2\in\{\leq,\geq,<,>,=,\neq\}$ confirms that there are no other Wilf equivalences between such consecutive patterns of relations, showing that the list is complete.
\end{proof}

\subsection{An application to  permutation patterns}

Inversion sequences avoiding the patterns $\left(\underline{\geq,>}\right)$ and $\left(\underline{>,\geq}\right)$ are closely related to permutations avoiding vincular patterns. We will exploit this connection to deduce Corollary~\ref{BaxterandPudwell} from Theorem~\ref{EquivIneq}(iv), using the following lemma.

\begin{lem}\label{lem:A200403} Let $\pi\in S_{n}$ and $e=\Theta(\pi)$. Then:
\begin{enumerate}[label=(\alph*),itemsep=1ex,leftmargin=1.5cm]
\item $e_{i}\geq e_{i+1}$ if and only if $\pi_{i}<\pi_{i+1}$;
\item $e_{i}> e_{i+1}$ if and only if there exists $j<i$ such that $\pi_{i}<\pi_{j}<\pi_{i+1}$.
\end{enumerate}
\end{lem}

\begin{proof}
Let $E_{i}=\left\{j\in[n]:j<i\textnormal{ and }\pi_{j}>\pi_{i}\right\}$, so that the entries of $e$ satisfy $e_{i}=\left|E_{i}\right|$.

To prove (a), first suppose that $\pi_{i}<\pi_{i+1}$. Then every $j<i+1$ such that $\pi_{j}>\pi_{i+1}$ satisfies $j<i$ and $\pi_{j}>\pi_{i}$. It follows that $E_{i+1}\subseteq E_{i}$, and so $e_{i}\geq e_{i+1}$.
Conversely, suppose that $\pi_{i}>\pi_{i+1}$. Then every $j<i$ such that $\pi_{j}>\pi_{i}$ satisfies $\pi_{j}>\pi_{i+1}$, so $E_{i}\subseteq E_{i+1}$. Furthermore, $i\in E_{i+1}\backslash E_{i}$, so this inclusion is strict and, consequently, $e_{i}<e_{i+1}$.

To prove (b), suppose that $e_{i}> e_{i+1}$. Then $\pi_{i}<\pi_{i+1}$  by part (a), and so $E_{i+1}\subseteq E_{i}$. If for every $j<i$ we had either $\pi_{j}<\pi_{i}$ or $\pi_{i+1}<\pi_{j}$, then $E_{i}\subseteq E_{i+1}$, from where $E_{i}=E_{i+1}$ and $e_{i}=e_{i+1}$, which is a contradiction. Hence, there exists $j<i$ such that $\pi_{i}<\pi_{j}<\pi_{i+1}$.

Conversely, if $j<i$ is such that $\pi_{i}<\pi_{j}<\pi_{i+1}$, then $E_{i+1}\subseteq E_{i}$ and $j\in E_{i}\backslash E_{i+1}$, so this inclusion is strict. Thus, $e_{i}>e_{i+1}$.
\end{proof}

The following result connects the vincular permutation patterns $\underline{124}3$ and $\underline{421}3$ to the consecutive patterns of relations $\left(\underline{>,\geq}\right)$ and $\left(\underline{\geq,>}\right)$, respectively. In the proof we will use the next two definitions.
For a permutation $\pi\in S_{n}$, define its {\em reverse} $\pi^{R}$ to be the permutation with entries $\pi^{R}_{i}=\pi_{n+1-i}$, and its {\em reverse-complement} $\pi^{RC}$ to be the permutation such that $\pi^{RC}_{i}=n+1-\pi^{R}_{i}$ for all~$i$.

\begin{prop}\label{prop:A200403}
Let $\pi\in S_{n}$ and $e=\Theta(\pi)$. Then:
\begin{enumerate}[label=(\alph*),itemsep=1ex,leftmargin=1.5cm]
\item $\pi$ avoids the vincular pattern $2\underline{134}$ if and only if $e$ avoids the consecutive pattern of relations $\left(\underline{>,\geq}\right)$. Consequently, $\left|S_{n}\left(\underline{124}3\right)\right|=\left|S_{n}\left(2\underline{134}\right)\right|=\left|\bI_{n}\left(\underline{>,\geq}\right)\right|$.
\item $\pi$ avoids the vincular pattern $3\underline{124}$ if and only if $e$ avoids the consecutive pattern of relations $\left(\underline{\geq,>}\right)$. Consequently, $\left|S_{n}\left(\underline{421}3\right)\right|=\left|S_{n}\left(3\underline{124}\right)\right|=\left|\bI_{n}\left(\underline{\geq,>}\right)\right|$.
\end{enumerate}
\end{prop}

\begin{proof}
The permutation $\pi$ contains the vincular pattern $2\underline{134}$ if and only if there exist indices $j<i$ such that $\pi_i<\pi_j<\pi_{i+1}<\pi_{i+2}$. By Lemma~\ref{lem:A200403}, these inequalities are equivalent to $e_i>e_{i+1}\ge e_{i+2}$, namely, to $e$ having an occurrence of $\left(\underline{>,\geq}\right)$ in position $i$. It follows that $\pi$ avoids $2\underline{134}$ if and only if $e$ avoids $\left(\underline{>,\geq}\right)$. Since the map $\pi\rightarrow\pi^{RC}$ induces a bijection between $S_{n}\left(\underline{124}3\right)$ and $S_{n}\left(2\underline{134}\right)$, this proves part~(a).

Similarly, $\pi$ contains the vincular pattern $3\underline{124}$ if and only if there exist indices $j<i$ such that $\pi_i<\pi_{i+1}<\pi_j<\pi_{i+2}$. By Lemma~\ref{lem:A200403}, this is equivalent $e$ having an occurrence of $\left(\underline{\ge,>}\right)$ in position $i$. Thus, $\pi$ avoids $3\underline{124}$ if and only if $e$ avoids $\left(\underline{\geq,>}\right)$. Since the map $\pi\rightarrow\pi^{R}$ induces a bijection between $S_{n}\left(\underline{421}3\right)$ and $S_{n}\left(3\underline{124}\right)$, part (b) follows.
\end{proof}

 From Theorem~\ref{EquivIneq}(iv) and Proposition~\ref{prop:A200403}, we deduce that $\left|S_{n}\left(\underline{124}3\right)\right|=\left|S_{n}\left(\underline{421}3\right)\right|$, which proves Corollary~\ref{BaxterandPudwell}. This result was originally conjectured by
Baxter and Pudwell~{\cite[Conjecture~17]{BaxterPudwell}},
and later proved by Baxter and Shattuck~{\cite[Corollary~11]{BaxterShattuck}} and by Kasraoui~\cite[Corollary~1.9(a)]{Kasraoui}, who showed that the vincular permutation patterns $\underline{124}3$ and $\underline{421}3$ are, in fact, strongly Wilf equivalent.

However, none of these proofs provides a simple bijection between $S_{n}\left(\underline{124}3\right)$ and $S_{n}\left(\underline{421}3\right)$. Next we present such a bijection. Its easiest description is at the level of inversion sequences. Note that, even though we proved in Proposition~\ref{prop:sergi_equivalences_consec_relations_2} that $\left(\underline{\geq,>}\right)$ and $\left(\underline{>,\geq}\right)$ are super-strongly Wilf equivalent, our proof did not give a bijection between the corresponding pattern-avoiding sets.

\begin{prop}\label{prop:A200403_biject}
For $n\geq 1$, there is an explicit bijection $\Phi:\bI_{n}\left(\underline{\geq,>}\right)\to\bI_{n}\left(\underline{>,\geq}\right)$ with the property that if $e'=\Phi(e)$, then $e'_n=e_n$.
\end{prop}

\begin{proof}
Given $e\in\bI_{n}\left(\underline{\geq,>}\right)$, define $\Phi(e)=e'$
to be the inversion sequence obtained by replacing each maximal
occurrence of a consecutive pattern $\underline{10^{r}}$ (for $r\ge2$) in $e$ with an occurrence of $\underline{1^{r}0}$.
In other words, each maximal subsequence of the form
$e_j>e_{j+1}=e_{j+2}=\dots =e_{j+r}$ (with $r\ge2$) becomes a subsequence
$e'_j=e'_{j+1}=\dots =e'_{j+r-1}>e'_{j+r}$, where $e'_j=e_j$ and $e'_{j+r}=e_{j+r}$.
Note that $\Phi$ acts on these subsequences in analogy to how the map $\Psi_{S}$, described in the schematic diagram in Figure~\ref{fig:Scheme_iii}, acted on blocks.

We now show that $e'\in\bI_{n}\left(\underline{>,\geq}\right)$. Since $e\in\bI_{n}\left(\underline{\geq,>}\right)$, it avoids $\left(\underline{>,>}\right)$ as well. Given that $\Phi$ removes all occurrences of $\left(\underline{>,=}\right)$ in $e$, it suffices to prove that $\Phi$ introduces no new occurrences of $\left(\underline{>,\geq}\right)$ at the edges of the changed subsequences; that is, that $e'_{j-1}e'_{j}e'_{j+1}$ and $e'_{j+r-1}e'_{j+r}e'_{j+r+1}$ are not occurrences of $\left(\underline{>,\geq}\right)$ for $j$ and $r$ as above.

We argue by contradiction. Suppose that $e'_{j-1}>e'_{j}\ge e'_{j+1}$. Then either $e_{j-1}=e'_{j-1}$ or $e_{j-1}=e'_{j}$. Thus, $e_{j-1}\geq e'_{j}=e_{j}>e_{j+1}$, so $e_{j-1}e_{j}e_{j+1}$ would have been an occurrence of $\left(\underline{\geq,>}\right)$.

Now, suppose that $e'_{j+r-1}>e'_{j+r}\ge e'_{j+r+1}$. Then $e_{j+r-1}=e_{j+r}=e'_{j+r}\geq e'_{j+r+1}\geq e_{j+r+1}$. However, at least one of these two inequalities must be strict because $e_{j+r}=e_{j+r+1}$ would contradict the maximality of $r$. But then $e_{j+r-1}e_{j+r}e_{j+r+1}$  would have been an occurrence of $\left(\underline{\geq,>}\right)$.

It is clear that $\Phi$ is a bijection, since the inverse map can be obtained by replacing each maximal
occurrence of a consecutive pattern $\underline{1^{r}0}$ (for $r\ge2$) with an occurrence of $\underline{10^{r}}$. Also, by construction, $\Phi$ preserves the last entry of~$e$.
\end{proof}

Combining Propositions~\ref{prop:A200403} and~\ref{prop:A200403_biject}, we see that the map $\pi\mapsto \Theta^{-1}\left(\Phi(\Theta(\pi^{RC}))\right)^R$ is a bijection between $S_{n}\left(\underline{124}3\right)$ and $S_{n}\left(\underline{421}3\right)$.

\section{Enumerative Results}\label{sec:enumerative_results}

In this section, we show that for consecutive patterns of relations $\left(\underline{R_{1},R_{2}}\right)$ in 14 of the 30 Wilf equivalence classes, the sequence $\left|\bI_{n}\left(\underline{R_{1},R_{2}}\right)\right|$ matches a sequence in~\cite{OEIS} that is known to enumerates another combinatorial object.
These results are summarized in Table~\ref{tab3}.

This leaves 16 Wilf equivalence classes for which the first few terms of the sequence $\left|\bI_{n}\left(\underline{R_{1},R_{2}}\right)\right|$ do not match any previously existing sequence in~\cite{OEIS}.
For 4 of these 16 classes, avoidance of a pattern of relations equates to avoidance of a single consecutive pattern.
Specifically, the patterns of relations $\left(\underline{<,=}\right)$, $\left(\underline{=,<}\right)$, $\left(\underline{>,=}\right)\stackrel{ss}{\sim}\left(\underline{=,>}\right)$ and $\left(\underline{>,>}\right)$ correspond to the consecutive patterns $\underline{011}$, $\underline{001}$, $\underline{100}\stackrel{ss}{\sim}\underline{110}$ and $\underline{210}$, respectively.
Recurrences for these 4 classes are known, as they were treated in the systematic study of consecutive patterns in inversion sequences~\cite{AuliElizalde}. The remaining 12 Wilf equivalence classes of consecutive patterns of relations, listed from least avoided to most avoided in inversion sequences of length~10, are:
\vspace{-6pt}
\begin{multicols}{4}
\noindent $\left(\underline{\leq,\leq}\right)$,\\
$\left(\underline{<,\neq}\right)$,\\
$\left(\underline{\neq,<}\right)$,\\
$\left(\underline{\leq,<}\right)$,\\
$\left(\underline{\neq,>}\right)$,\\
$\left(\underline{<,>}\right)$,\\
$\left(\underline{<,\leq}\right)$,\\
$\left(\underline{>,\neq}\right)$,\\
$\left(\underline{>,<}\right)$,\\
$\left(\underline{=,\leq}\right)$,\\
$\left(\underline{\leq,=}\right)$,\\
$\left(\underline{\geq,=}\right)\stackrel{ss}{\sim}\left(\underline{=,\geq}\right)$.
\end{multicols}\vspace{-5pt}
For each pattern $\left(\underline{R_{1},R_{2}}\right)$ above, even though we do not have closed formulas for $\left|\bI_{n}\left(\underline{R_{1},R_{2}}\right)\right|$, we can obtain recurrences to compute these quantities. The recurrences are similar to the one in Proposition~\ref{prop:A200403_2}, and can be proved using analogous arguments.

\subsection{The pattern $\left(\underline{\leq,\neq}\right)$}\label{subsec:A040000}

The following lemma gives a very simple description of the set $\bI_{n}\left(\underline{\leq,\neq}\right)$. The notation $a^s$ indicates repetition $s$ times of the entry $a$.

\begin{lem}\label{lem:A040000}
For $n\geq 2$,
$$\bI_{n}\left(\underline{\leq,\neq}\right)=\left\{0^{n},01^{n-1}\right\}.$$
\end{lem}
\begin{proof} We proceed by induction on $n$. The statement is trivially true for the base case $n=2$.
Now let $n\ge3$, suppose the statement holds for $n-1$, and let $e\in\bI_{n}\left(\underline{\leq,\neq}\right)$. Then $e_{1}e_{2}\ldots e_{n-1}\in\bI_{n-1}\left(\underline{\leq,\neq}\right)$, so either $e=0^{n-1}$ or $e=01^{n-2}$. In any case, $e_{n-2}\leq e_{n-1}$, so it must be that $e_{n-1}=e_{n}$. Hence, either $e=0^{n}$ or $e=01^{n-1}$. Since both of these inversion sequences avoid $\left(\underline{\leq,\neq}\right)$, we deduce that $\bI_{n}\left(\underline{\leq,\neq}\right)=\left\{0^{n},01^{n-1}\right\}$.
\end{proof}

It follows from Lemma~\ref{lem:A040000} that
$\left|\bI_{1}\left(\underline{\leq,\neq}\right)\right|=1$ and $\left|\bI_{n}\left(\underline{\leq,\neq}\right)\right|=2$ for $n\ge2$.

\subsection{The pattern $\left(\underline{\leq,\geq}\right)$}\label{subsec:A000027}

The next lemma provides a characterization of the set $\bI_{n}\left(\underline{\leq,\geq}\right)$ in terms of a monotonicity condition. This will be a recurrent idea in this paper, as it is often convenient to describe the inversion sequences in $\bI_{n}\left(\underline{R_{1},R_{2}}\right)$ in terms of a monotonicity or unimodal condition, in order to enumerate them.

\begin{lem}\label{lem:A000027_1}
For $n\geq 1$,
$$\bI_{n}\left(\underline{\leq,\geq}\right)=\{e\in\bI_{n} : e_2<e_3<\dots<e_n\}.$$
\end{lem}

\begin{proof}
The inclusion to the left is immediate. To prove the inclusion to the right, let $e\in\bI_{n}\left(\underline{\leq,\geq}\right)$. We will show that $e_{j}<e_{j+1}$ for $2\leq j<n$ by induction on~$j$.
For the base case $j=2$, note that $e_1=0\le e_2$ implies $e_2<e_3$, otherwise $e_1e_2e_3$ would be an occurrence of $\left(\underline{\leq,\geq}\right)$. Now suppose that $e_{j}<e_{j+1}$ for some $2\le j<n-1$. Then $e_{j+1}<e_{j+2}$, because otherwise $e_je_{j+1}e_{j+2}$ would be an occurrence of $\left(\underline{\leq,\geq}\right)$.
\end{proof}

\begin{prop}\label{prop:A000027}
For $n\geq 1$, $\left|\bI_{n}\left(\underline{\leq,\geq}\right)\right|=n$.
\end{prop}

\begin{proof}
It follows from Lemma~\ref{lem:A000027_1} and the definition of inversion sequence that $\bI_{n}\left(\underline{\leq,\geq}\right)$ is the set of sequences $e_1\dots e_n$ with $0=e_1\le e_2<e_3<\dots<e_n\le n-1$. Disregarding the forced entry $e_1=0$, this condition in equivalent to \begin{equation}\label{eq:le}
0\le e_2 \le e_3-1 \le e_4-2 \le \dots \le e_n-(n-2)\le 1,
\end{equation}
and so these entries are determined by the choice of which one of the $n$ inequalities in~\eqref{eq:le} is strict. In other words, these entries are given by
$$e_i=\begin{cases} i-2 & \text{if }2\le i\le j,\\ i-1 & \text{if }j< i \le n \end{cases}$$
for some fixed $1\le j\le n$. Thus, there are $n$ such sequences.
\end{proof}

\subsection{The pattern $\left(\underline{\geq,\neq}\right)$}\label{subsec:A000124}
The next lemma characterizes inversion sequences avoiding this pattern.

\begin{lem}\label{lem:A000124}
For $n\ge1$,
$$\bI_{n}\left(\underline{\geq,\neq}\right)=\{e\in\bI_{n}:
e_{1}< e_{2}<\dots < e_{j}\geq e_{j+1}=e_{j+2}=\dots =e_{n} \text{ for some }1\leq j\leq n\}.$$
\end{lem}

\begin{proof}
The inclusion to the left is straightforward. For the inclusion to the right, let $e\in\bI_{n}\left(\underline{\geq,\neq}\right)$, and let $j$ be the largest integer such that $e_{1}<e_{2}<\dots <e_{j}$.
If $j<n$, then $e_{j}\geq e_{j+1}$ by construction, and since $e$ avoids $\left(\underline{\geq,\neq}\right)$, it must be that $e_{j+1}=e_{j+2}$ (if these entries are defined). Repeating the same argument, we have $e_{j+1}=e_{j+2}=\dots =e_{n}$.
\end{proof}

Next we obtain a formula for the number of inversion sequences that avoid $\left(\underline{\geq,\neq}\right)$. This sequence
appears as A000124 in~\cite{OEIS}; it is sometimes referred to as {\em central polygonal numbers} or as the {\em lazy caterer's sequence}.

\begin{prop}
\label{prop:A000124} For $n\geq 1$, $\left|\bI_{n}\left(\underline{\geq,\neq}\right)\right| = \binom{n}{2}+1$.
\end{prop}

\begin{proof}
Let $\mathcal{A}_{n}$ be the collection of subsets of $\{0,1,\dots,n-2\}$ with at most two elements. Clearly, $\left|\mathcal{A}_{n}\right|=\binom{n-1}{2}+(n-1)+1=\binom{n}{2}+1$. We define a bijection $\Gamma:\bI_{n}\left(\underline{\geq,\neq}\right)\rightarrow \mathcal{A}_{n}$ by letting, for each $e\in\bI_{n}\left(\underline{\geq,\neq}\right)$,
$$\Gamma(e) = \begin{cases}
\{\max_i\{e_i\},e_n\}  & \text{if } e_{n}\neq n-1,\\
\emptyset & \text{otherwise},
    \end{cases}
$$
where $\{a,a\}$ is simply the set $\{a\}$. If $j$ is the index such that $e_{1}< e_{2}<\dots < e_{j}\geq e_{j+1}=e_{j+2}=\dots =e_{n}$,
which exists by Lemma~\ref{lem:A000124} and is unique by definition, then $e_i=i-1$ for $1\le i\le j$, so the above definition can be restated as
$\Gamma(e) = \{e_j,e_n\} $ if $j<n$ and $\Gamma(e)=\emptyset$ otherwise.

To see that $\Gamma$ is a bijection, we describe its inverse. Given any $A\in\mathcal{A}_n$, we have
$$\Gamma^{-1}(A)=\begin{cases}
012\dots abb\dots b& \text{if } A=\{a,b\} \text{ with }a> b,\\
012\dots aa\dots a& \text{if } A=\{a\},\\
012\dots(n-1) & \text{if }A=\emptyset.
    \end{cases}
$$
This concludes the proof.
\end{proof}

\begin{exa} Applying the map $\Gamma$ from the proof of Proposition~\ref{prop:A000124}  to $e=012322\in\bI_{6}\left(\underline{\geq,\neq}\right)$, we get $\Gamma(e)=\{3,2\}$. Similarly, if $e=012344\in\bI_{6}\left(\underline{\geq,\neq}\right)$, then $\Gamma(e)=\{4\}$. \end{exa}

One can also obtain the formula in Proposition~\ref{prop:A000124} as follows.
Martinez and Savage~\cite[Observation~11]{MartinezSavageII} note that the set $\bI_{n}\left(\geq,\neq,-\right)$ is characterized by the same inequalities from Lemma~\ref{lem:A000124}, so we deduce that $\bI_{n}\left(\underline{\geq,\neq}\right)=\bI_{n}\left(\geq,\neq,-\right)$. Moreover, they show in~\cite[Theorem~9]{MartinezSavageII} that $\bI_{n}\left(\geq,\neq,-\right)=\Theta\left(S_{n}(213,321)\right)$, where $\Theta$ is given by \eqref{eq:theta_bijection} and $S_{n}(213,321)$ denotes the set of permutations in $S_n$ avoiding both of the (classical) permutation patterns $213$ and $321$. It was shown by Simion and Schmidt \cite[Proposition~11]{SimionSchmidt} that
$\left|S_{n}(213,321)\right|=\binom{n}{2}+1$. We conclude that
\[
\left|\bI_{n}\left(\underline{\geq,\neq}\right)\right| = \left|\bI_{n}\left(\geq,\neq,-\right)\right| = \left|S_{n}(213,321)\right| = \binom{n}{2}+1.
\]

\subsection{The pattern $\left(\underline{\geq,\leq}\right)$}\label{subsec:A000045}

The next lemma characterizes the set $\bI_{n}\left(\underline{\geq,\leq}\right)$ as strictly unimodal inversion sequences.

\begin{lem}\label{lem:A000045}
For $n\ge1$,
$$
\bI_{n}\left(\underline{\geq,\leq}\right)=\{e\in\bI_{n}: e_{1}< e_{2}<\dots < e_{j}\geq e_{j+1}>e_{j+2}>\dots >e_{n} \text{ for some }1\leq j\leq n\}.
$$
\end{lem}

\begin{proof} The inclusion to the left is immediate, so let us prove the inclusion to the right. Suppose that  $e\in\bI_{n}\left(\underline{\geq,\leq}\right)$, and let $j$ be the smallest integer such that $e_{j}\ge e_{j+1}$, or $j=n$ if no such integer exists. Since $e$ avoids $\left(\underline{\geq,\leq}\right)$, it must be that $e_{j+1}>e_{j+2}$ (if these entries are defined). Repeating the same argument, we see that $e_{j+1}>e_{j+2}>\dots >e_{n}$. On the other hand, $e_{j-1}< e_{j}$ (if $j\ge2$) by definition of $j$. Since $e$ avoids $\left(\underline{\geq,\leq}\right)$, we must then have
$e_{1}<e_{2}<\dots <e_{j}$. We conclude that $e$ satisfies the stated unimodality condition.
\end{proof}

Let $\mathcal{C}_n$ be the set of compositions of $n$ with parts $1$ and $2$, that is, sequences $(a_1,a_2,\dots,a_j)$ such that $a_1+\dots+a_j=n$
and $a_i\in\{1,2\}$ for $1\le i\le j$. It is well-known that $|\mathcal{C}_n|=F_{n+1}$, the $(n+1)$th Fibonacci number, defined by the recurrence $F_{0}=0$, $F_{1}=1$ and $F_{n}=F_{n-1}+F_{n-2}$ for $n\geq 2$. The integer sequence sequence $F_{n}$ is listed as A000045 in~\cite{OEIS}.

\begin{prop} There is a bijective correspondence between $\bI_{n}\left(\underline{\geq,\leq}\right)$ and $\mathcal{C}_{n}$. In particular,
$\left|\bI_{n}\left(\underline{\geq,\leq}\right)\right|=F_{n+1}$.
\end{prop}

\begin{proof}
Given $e\in \bI_{n}\left(\underline{\geq,\leq}\right)$, we know by Lemma~\ref{lem:A000045} that
$$e_{1}< e_{2}<\dots < e_{j}\geq e_{j+1}>e_{j+2}>\dots >e_{n}$$ for some $j$. Since $e$ is an inversion sequence, it follows that $e_{i}=i-1$ for $1\leq i\leq j$.

Define $f(e)=(a_1,a_2,\dots,a_j)$, where $a_i$ is the number of entries in $e$ that are equal to $i-1$, for $1\le i\le j$.
Note that $f(e)\in \mathcal{C}_{n}$ because, for each $1\le i\le j$, the sequence $e$ has one or two entries equal to $i-1$, and it has no entries larger than $j-1$.

Let us show that $f$ is a bijection by describing its inverse. Given $(a_1,a_2,\dots,a_j)\in \mathcal{C}_{n}$, we
recover the unique $e\in \bI_{n}\left(\underline{\geq,\leq}\right)$ such that $f(e)=(a_1,a_2,\dots,a_j)$ as follows.
First, set $e_{i}=i-1$ for $1\le i\le j$. To define the remaining entries, let $i$ range from $1$ to $j$, and whenever $a_i=2$, set the rightmost entry of $e$ that has not been defined equal to $e_i$.
\end{proof}

\begin{figure}
\begin{centering}
\includegraphics[scale=0.65]{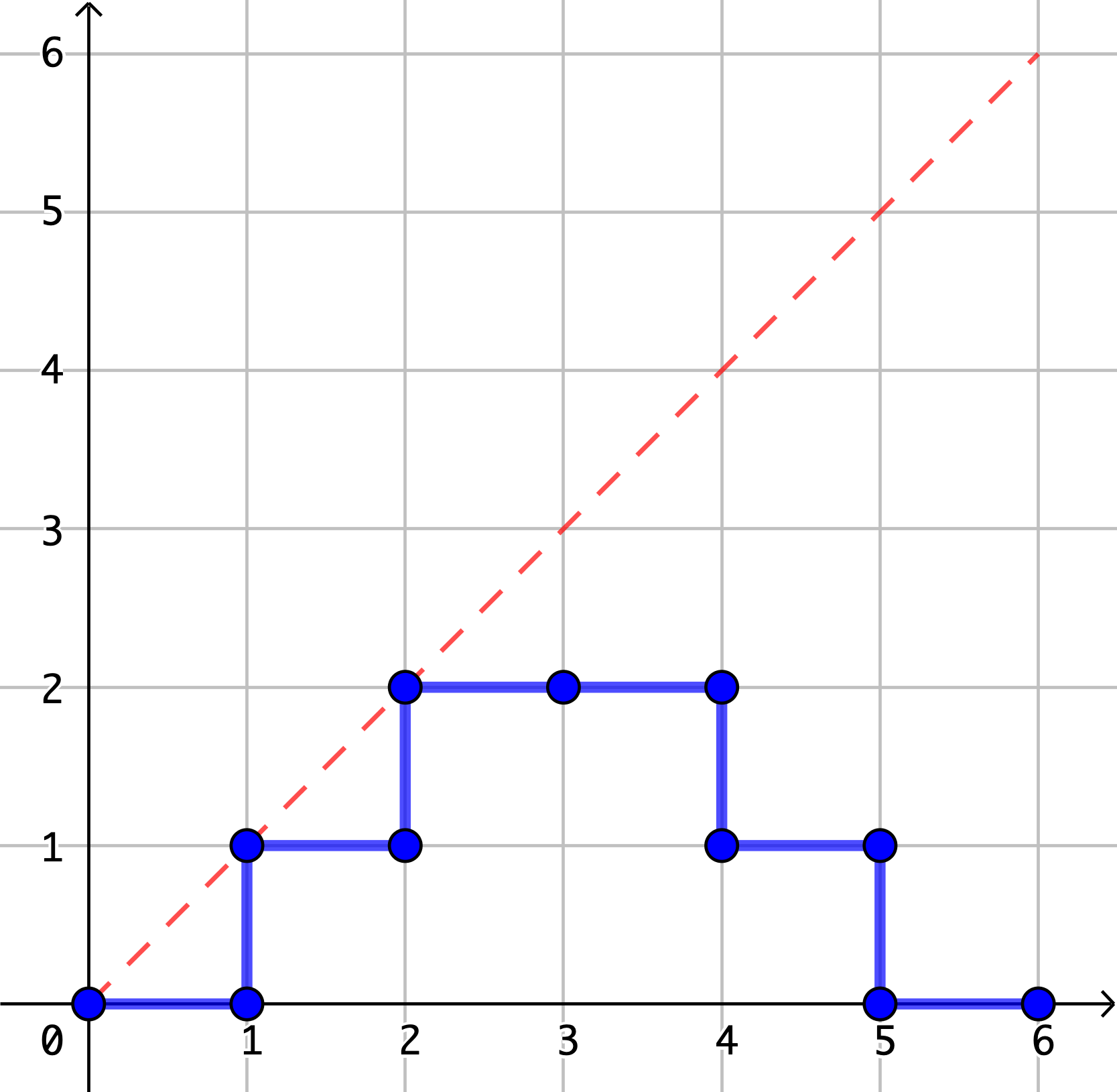}
\hspace*{1.5cm} 
\includegraphics[scale=0.65]{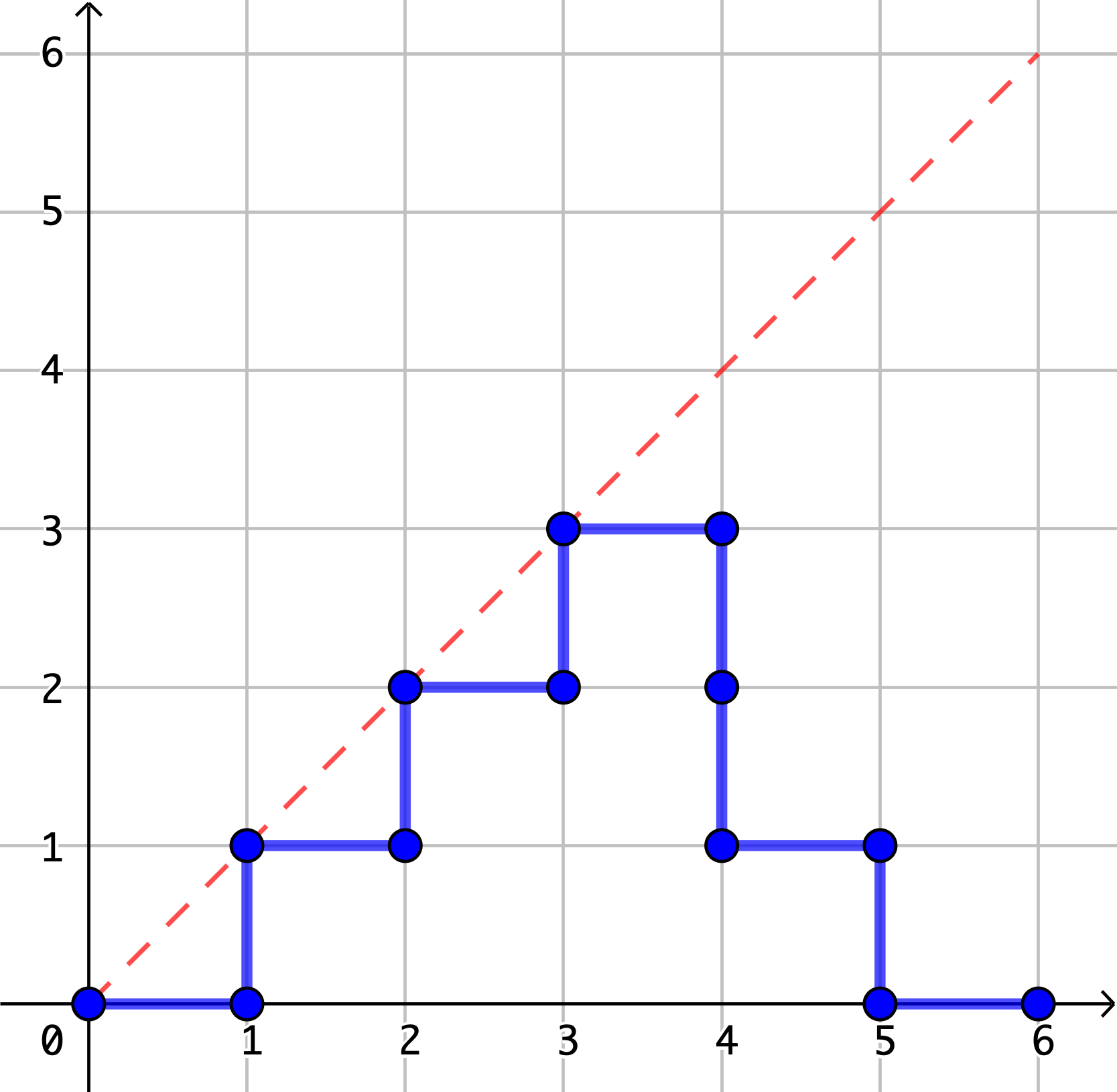}\par
\end{centering}
\caption{Two inversion sequences $e\in\bI_{6}\left(\underline{\geq,\leq}\right)$: $e=012210$ (left) and $e=012310$ (right).} \label{fig:cusp}
\end{figure}

\begin{exa} Applying the bijection $f$ from the above proof to the inversion sequences in Figure~\ref{fig:cusp}, we get $f(012210)=(2,2,2)$ and
$f(012310)=(2,2,1,1)$. \end{exa}

\subsection{The pattern $\left(\underline{\neq,\leq}\right)$}\label{subsec:A000071}
As in previous subsections, we start by characterizing the inversion sequences avoiding this pattern.

\begin{lem}\label{lem:A000071}
For $n\ge1$,
$$\bI_{n}\left(\underline{\neq,\leq}\right)=\{e\in\bI_{n}:
e_{1}= e_{2}=\dots = e_{j-1}<e_{j}>e_{j+1}>\dots >e_{n} \text{ for some }2\leq j\leq n+1\}.$$
\end{lem}

\begin{proof}
The inclusion to the left is straightforward. For the inclusion to the right,
take $e\in\bI_{n}\left(\underline{\neq,\leq}\right)$ and let
$2\leq j\leq n+1$ be the largest integer such that $0=e_{1}=e_{2}=\dots =e_{j-1}$. If $j\le n$, then $e_{j-1}<e_{j}$. Furthermore, if $j\le n-1$, then $e_{j}>e_{j+1}$ because $e$ avoids the pattern $\left(\underline{\neq,\leq}\right)$. The same argument shows that $e_{j}>e_{j+1}>\dots >e_{n}$.
\end{proof}

Martinez and Savage prove in \cite[Theorem~12]{MartinezSavageII} that the inversion sequences in $\bI_{n}\left(\neq,\leq,-\right)$ are characterized by the same inequalities from Lemma~\ref{lem:A000071}, so we deduce that $\bI_{n}\left(\underline{\neq,\leq}\right)=\bI_{n}\left(\neq,\leq,-\right)$. They also show that $\left|\bI_{n}\left(\neq,\leq,-\right)\right|=F_{n+2}-1$ for $n\geq 1$.
This sequence is listed as A000071 in~\cite{OEIS}. The next corollary follows.

\begin{cor}\label{cor:A000071} For $n\geq 1$, $\left|\bI_{n}\left(\underline{\neq,\leq}\right)\right|=\left|\bI_{n}\left(\neq,\leq,-\right)\right|=F_{n+2}-1$.
\end{cor}

\subsection{The patterns $\left(\underline{\geq,<}\right)\stackrel{ss}{\sim}\left(\underline{<,\geq}\right)
\sim\left(\underline{\neq,\geq}\right)$}\label{subsec:A000079}
We have the following characterizations of inversion sequences avoiding these three patterns.

\begin{lem}\label{lem:A000079}
For $n\ge1$,
\begin{align*}\bI_{n}\left(\underline{<,\geq}\right)=\bI_{n}\left(\underline{\neq,\geq}\right)=\{e\in\bI_{n}:
e_{1}= e_{2}=\dots = e_{j}<e_{j+1}<\dots <e_{n}\text{ for some }1\leq j\leq n\},\\
\bI_{n}\left(\underline{\geq,<}\right)=\{e\in\bI_{n}:
e_{1}< e_{2}<\dots <e_{j}\ge e_{j+1}\ge \dots \ge e_{n}\text{ for some }1\leq j\leq n\}.
\end{align*}
\end{lem}

\begin{proof}
We showed in the proof of Proposition~\ref{prop:wilf_equiv_equivalences_consec_relations} that $\bI_{n}\left(\underline{<,\geq}\right)=\bI_{n}\left(\underline{\neq,\geq}\right)$.
Given $e\in\bI_{n}\left(\underline{<,\geq}\right)$, let $1\leq j\leq n$ be the largest integer such that $0=e_{1}=e_{2}=\dots =e_{j}$. If $j<n$, then $e_{j}<e_{j+1}$, and if $j<n-1$, then $e_{j+1}<e_{j+2}$ because $e$ avoids the pattern $\left(\underline{<,\geq}\right)$. An inductive argument shows that $e_{j}<e_{j+1}<\dots <e_{n}$.

The characterization of $e\in\bI_{n}\left(\underline{\geq,<}\right)$ can be obtained similarly by letting $j$ be the largest integer such that $e_{1}<e_{2}<\dots <e_{j}$. Alternatively, it follows from the fact, shown in the proof of Proposition~\ref{prop:complements_equivalences_consec_relations}(i$_1$), that the map $e\rightarrow e^{C}$ sending each inversion sequence to its complement
induces a bijection between $\bI_{n}\left(\underline{\geq,<}\right)$ and $\bI_{n}\left(\underline{<,\geq}\right)$.
\end{proof}

 Martinez and Savage show in \cite[Theorem~15]{MartinezSavageII} that the set $\bI_{n}\left(\neq,\leq,-\right)$ is also characterized
 by the same condition as in the first part of Lemma~\ref{lem:A000079}, and so $\bI_{n}\left(\underline{<,\geq}\right)=\bI_{n}\left(\underline{\neq,\geq}\right)=\bI_{n}\left(\neq,\leq,-\right)$. In addition, there is a bijection between this set and the set of all subsets of $[n-1]$, obtained by mapping an inversion sequence to the set of its nonzero values. The next result now follows using Theorem~\ref{EquivIneq}(i).

\begin{cor}\label{cor:A000079} For $n\geq 1$, $\left|\bI_{n}\left(\underline{\geq,<}\right)\right| = \left|\bI_{n}\left(\underline{<,\geq}\right)\right| = \left|\bI_{n}\left(\underline{\neq,\geq}\right)\right| = 2^{n-1}$.
\end{cor}

\subsection{The pattern $\left(\underline{\neq,\neq}\right)$}\label{subsec:A000085}
Let $\Inv_n$ denote the set of involutions of $[n]$, that is, permutations $\pi\in S_n$ such that $\pi^{-1}=\pi$.
We will show that there is a bijection between $\bI_n\left(\underline{\neq,\neq}\right)$ and $\Inv_n$.

\begin{prop}\label{prop:A000085}
There is a bijection $\Upsilon:\bI_{n}\left(\underline{\neq,\neq}\right)\to\Inv_{n}$. In particular, the number of inversion sequences avoiding $\left(\underline{\neq,\neq}\right)$ satisfies the recurrence
\begin{equation}\label{eq:recA000085}\left|\bI_n\left(\underline{\neq,\neq}\right)\right|=\left|\bI_{n-1}\left(\underline{\neq,\neq}\right)\right|+(n-1)\left|\bI_{n-2}\left(\underline{\neq,\neq}\right)\right|
\end{equation}
for $n\ge2$, with initial conditions $\left|\bI_0\left(\underline{\neq,\neq}\right)\right|=\left|\bI_1\left(\underline{\neq,\neq}\right)\right|=1$, and its exponential generating function is
\[
\sum_{n\geq 0}\left|\bI_{n}\left(\underline{\neq,\neq}\right)\right|\frac{z}{n!}=\exp \left(z+\frac{z^{2}}{2}\right).
\]
\end{prop}

\begin{proof}
We start by proving the recurrence~\eqref{eq:recA000085}, which will also inform the construction of the bijection $\Upsilon$. The initial conditions are trivially satisfied.
Let $n\ge2$, and let $e\in \bI_{n}\left(\underline{\neq,\neq}\right)$.
If $e_{n-1}=e_n$, then $e$ is obtained from an arbitrary inversion sequence in $\bI_{n}\left(\underline{\neq,\neq}\right)$ by repeating the last entry, so there are $\left|\bI_{n-1}\left(\underline{\neq,\neq}\right)\right|$ inversion sequences with $e_{n-1}=e_n$. If $e_{n-1}\ne e_n$, then we must have $e_{n-2}=e_{n-1}$ (unless $n=2$, in which case $e=01$), since $e$ avoids $\left(\underline{\neq,\neq}\right)$. In this case, $e$ is obtained from an arbitrary
inversion sequence in $\bI_{n-2}\left(\underline{\neq,\neq}\right)$ by repeating the last entry, and then appending any element from $\{0,1,\dots,n-1\}\setminus\{e_{n-2}\}$, for which there are $n-1$ choices. Thus, there are $(n-1)\left|\bI_{n-2}\left(\underline{\neq,\neq}\right)\right|$ inversion sequences with $e_{n-1}\ne e_n$. The recurrence is now proved.

It is well known that the number of involutions $|\Inv_n|$ satisfies the same recurrence. It follows that $\left|\bI_n\left(\underline{\neq,\neq}\right)\right|=|\Inv_n|$, and that their exponential generating function is $\exp \left(z+\frac{z^{2}}{2}\right)$.

To construct an explicit bijection $\Upsilon$, let $e\in\bI_{n}\left(\underline{\neq,\neq}\right)$, and define $\Upsilon(e)\in S_{n}$ recursively as follows. For $n=0$, define the image of the empty inversion sequence to be the empty permutation; for $n=1$, define $\Upsilon(0)=1$. For $n\geq 2$, let $i\in[n]$ be such that $e_{n}-e_{n-1}\equiv i\pmod{n}$.
\begin{itemize}
\item If $i=n$, let $\sigma=\Upsilon\left(e_{1}e_{2}\ldots e_{n-1}\right)$, and define
$\Upsilon(e)=\sigma_{1}\sigma_{2}\ldots \sigma_{n-1}\,n$.
\item If $i\neq n$, let $\sigma$ be the permutation of $[n-1]\backslash\left\{i\right\}$ with reduction $\Upsilon\left(e_{1}e_{2}\ldots e_{n-2}\right)$, and define
$\Upsilon(e)=\sigma_{1}\sigma_{2}\ldots \sigma_{i-1}\,n\,\sigma_{i}\sigma_{i+1}\ldots \sigma_{n-2}\,i$.\qedhere
\end{itemize}
\end{proof}

\begin{exa}\label{exa:A000085} Consider $e=00114$. Computing $\Upsilon\left(e_{1}e_{2}\ldots e_{i}\right)$ recursively yields: $\Upsilon(0)=1$, $\Upsilon(00)=12$, $\Upsilon(001)=321$, $\Upsilon(0011)=3214$, and $\Upsilon(e)=42513$.
\end{exa}

Both maps $\Upsilon^{-1}$ (from the proof of Proposition~\ref{prop:A000085}) and $\Theta$ (from Equation~\eqref{eq:theta_bijection}) give encodings of involutions as inversion sequences, but they do not coincide in general. For instance, if $\pi=42513$, then $\Upsilon^{-1}(\pi)=00114\neq 01032=\Theta(\pi)$.

\subsection{The pattern $\left(\underline{\leq,>}\right)$}\label{subsec:A000108}
We show that inversion sequences avoiding this pattern are those that are weakly increasing, and that they are in bijection with Dyck paths.

\begin{lem}\label{lem:A000108}
For $n\ge1$,
$$\bI_{n}\left(\underline{\leq,>}\right)=\{e\in\bI_{n}:
e_{1}\leq e_{2}\leq\dots\leq e_{n}\}.$$
\end{lem}

\begin{proof}
If $e\in \bI_{n}\left(\underline{\leq,>}\right)$, the fact that $0=e_1\le e_2$ implies that $e_2\le e_3$, assuming these entries are defined. The remaining inequalities follow by iterating the same argument.
\end{proof}

A {\em Dyck path} of semilength $n$ is a lattice path in $\mathbb{Z}^{2}$ from $(0,0)$ to $(n,n)$ consisting of
horizontal steps $N=(0,1)$ and vertical steps $E=(1,0)$, which never goes above the line $y=x$.
Denoting by $\mathcal{D}_n$ the set of such paths, it is well-known that $\left|\mathcal{D}_{n}\right|=C_{n}=\frac{1}{n+1}\binom{2n}{n}$, the {\em $n$-th Catalan number}, and that
\begin{equation}\label{eq:Catalan} C(z):=\sum_{n\ge0}C_nz^n=\frac{1-\sqrt{1-4z}}{2z}.
\end{equation}
The sequence $C_{n}$ is one of the most ubiquitous sequences in enumerative combinatorics~\cite{StanleyCatalan}, and it is listed as A000108 in~\cite{OEIS}.

\begin{prop}\label{prop:A000108} Let $n\geq 1$ and $0\leq k<n$. Then
\[
\left|\bI_{n}\left(\underline{\leq,>}\right)\right|=C_{n}=\frac{1}{n+1}\binom{2n}{n}.
\]
\end{prop}

\begin{proof}
Using Lemma~\ref{lem:A000108}, we obtain a straightforward bijection  between $\bI_{n}\left(\underline{\leq,>}\right)$ and $\mathcal{D}_{n}$ by
appending $n-e_n$ steps $N$ to our usual representation of an inversion sequence $e\in\bI_{n}\left(\underline{\leq,>}\right)$ as an underdiagonal lattice path, see Figure~\ref{fig:CatalanPath}.
\end{proof}

\begin{figure}[htb]
	\noindent \begin{centering}
	\includegraphics[scale=0.6]{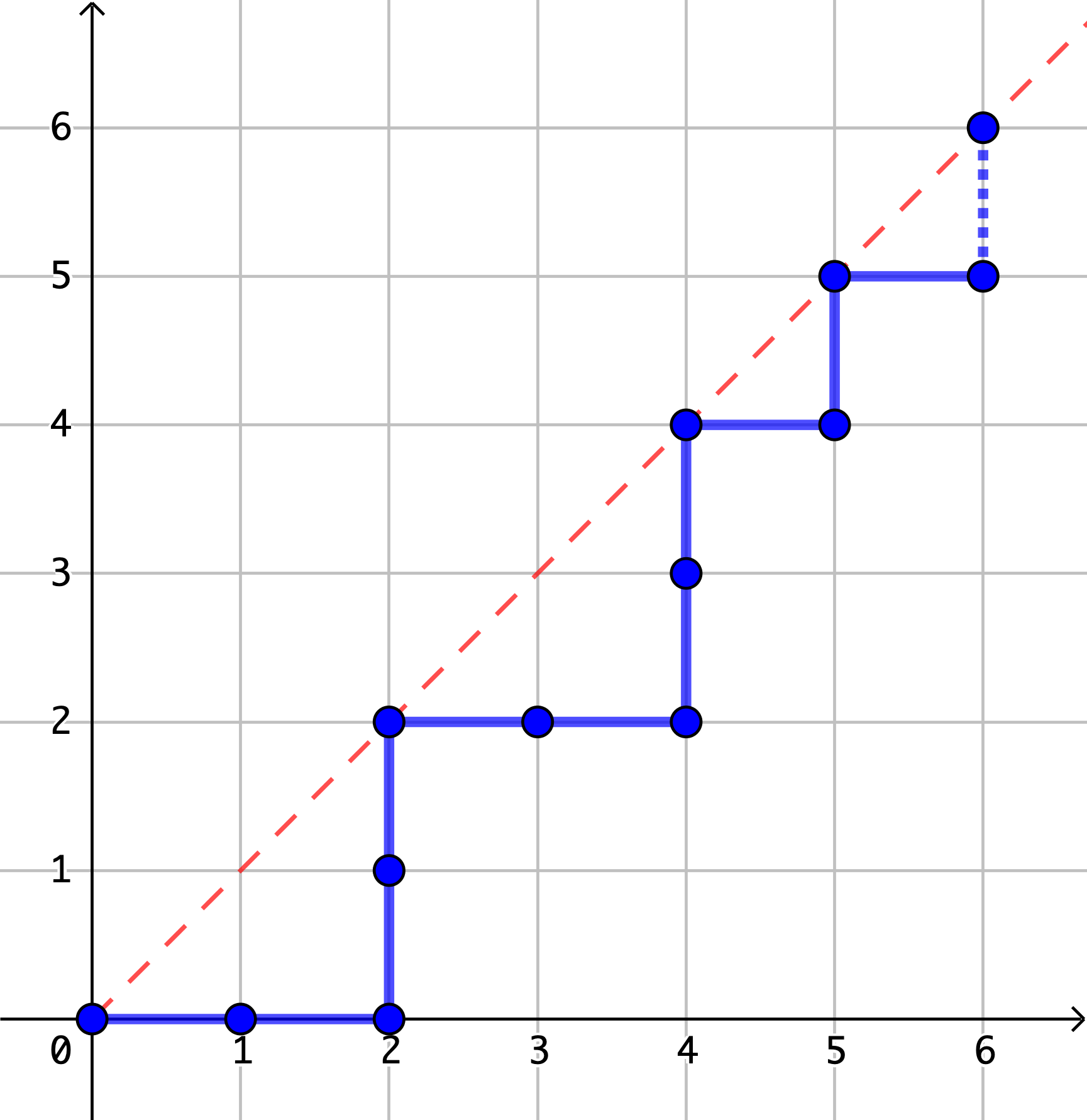}
		\par\end{centering}

	\protect\caption{Representation of $e=001123\in\bI_{6}\left(\underline{\leq,>}\right)$ as an underdiagonal lattice path. Adding the last vertical step we obtain the Dyck path $EENNEENNENEN$.\label{fig:CatalanPath}}
\end{figure}

\subsection{The pattern $\left(\underline{>,\leq}\right)$}\label{subsec:A071356}
We characterize inversion sequences avoiding this pattern in terms of an asymmetric unimodality condition.

\begin{lem}\label{lem:A071356}
For $n\geq 1$,
$$\bI_{n}\left(\underline{>,\leq}\right)=\{e\in\bI_{n}:
e_{1}\leq e_{2}\leq\dots\leq e_{j}>e_{j+1}>\dots >e_{n}\text{ for some }1\leq j\leq n\}.$$
\end{lem}

\begin{proof}
Given $e\in\bI_{n}\left(\underline{>,\leq}\right)$, let $1\le j\le n$ be the smallest index such that $e_{j}>e_{j+1}$, or let $j=n$ if there is no such index. Since $e$ avoids $\left(\underline{>,\leq}\right)$, we must have $e_{j}>e_{j+1}>\dots >e_{n}$. The characterization follows.
\end{proof}

Martinez and Savage show in~\cite[Section~2.19]{MartinezSavageII} that inversion sequences in $\bI_{n}\left(>,\leq,-\right)$ are those satisfying the same inequalities from Lemma~\ref{lem:A071356}, and so
\begin{equation}\label{eq:Iclassical}
\bI_{n}\left(\underline{>,\leq}\right)=\bI_{n}\left(>,\leq,-\right).
\end{equation} They also conjecture
that $\left|\bI_{n}\left(>,\leq,-\right)\right|$ is given by sequence A071356 from~\cite{OEIS}; equivalently, that Theorem~\ref{carlaConjecture} holds. To prove their conjecture, we introduce certain lattice paths.

\begin{defin}\label{def:marked} A {\em marked Dyck path} $P$ is an underdiagonal lattice path from $(0,0)$ to some point in the diagonal, with horizontal steps $E=(1,0)$ and two possible kinds of vertical steps $(0,1)$, denoted by $N$ and $N^*$ (the latter are called {\em marked} steps). Denoting by
$E(P)$, $N(P)$ and $N^*(P)$, the number of $E$, $N$ and $N^*$ steps in $P$, respectively, the size of $P$ is defined as $N^{*}(P)+E(P)=N(P)+2N^*(P)$. Let $\mathcal{P}_{n}$ denote the set of marked Dyck paths of size $n$.

If a marked Dyck path $P$ has at least one $N^{*}$ step in the last run of vertical steps, then we say that $P$ has a {\em marked tail}. Otherwise, we say that $P$ has an {\em unmarked tail}. We denote by $\mathcal{R}_{n}$ the set of paths in $\mathcal{P}_n$ with an unmarked tail.
\end{defin}

\begin{exa} The marked Dyck path $P_1=ENEN^{*}$ has size $3$ and a marked tail, so $P_1\in\mathcal{P}_3\setminus\mathcal{R}_3$. On the other hand, $P=ENEN^{*}EEN^{*}EENNN$ has size $8$ and an unmarked tail, so $P\in\mathcal{R}_{8}$. These paths are drawn in Figure~\ref{fig:carla_conj_proof2}, with a $*$ indicating which vertical steps are $N^*$ steps.
\end{exa}

\begin{figure}[htb]
	\noindent \begin{centering}
\includegraphics[scale=0.6]{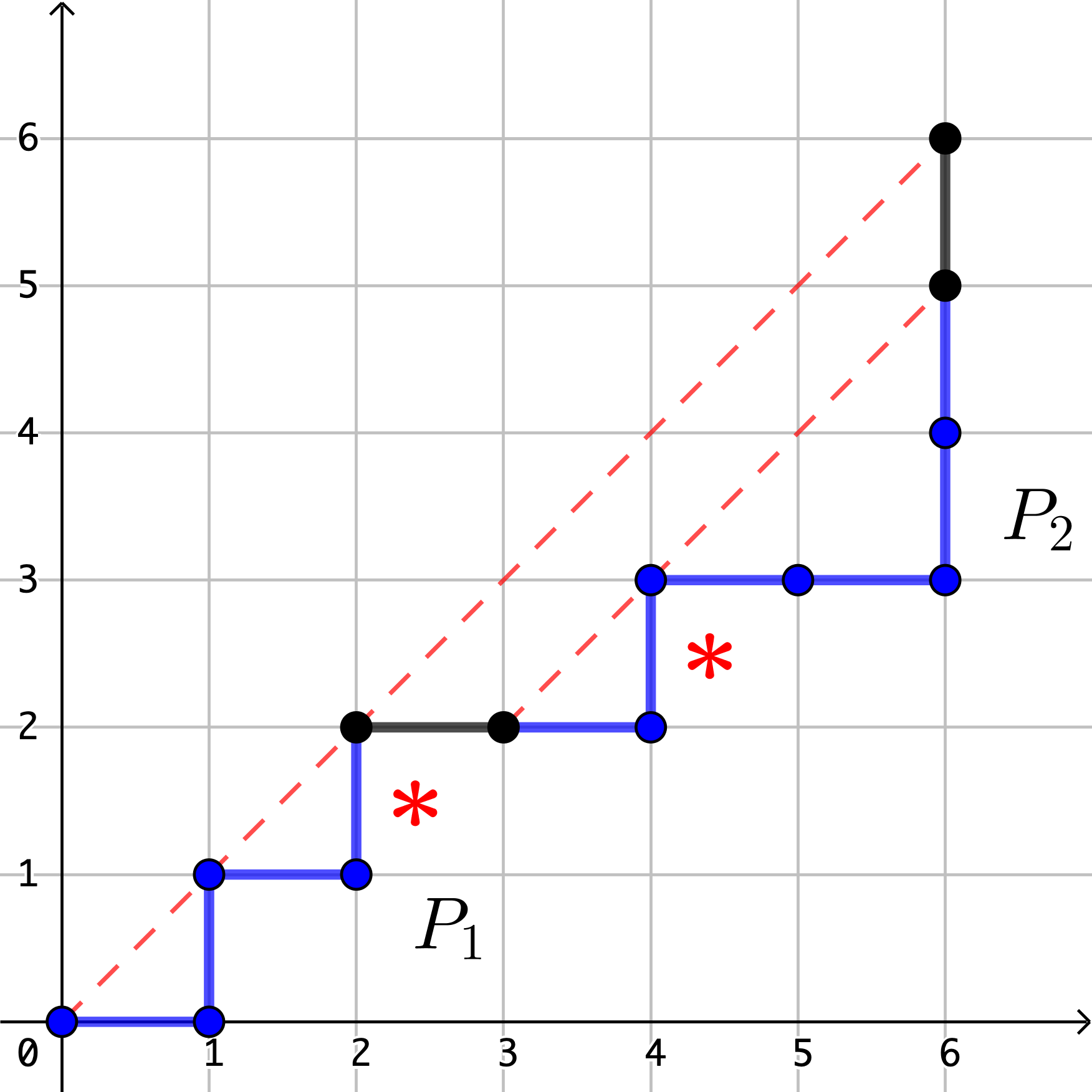}
		\par\end{centering}
\protect\caption{The marked Dyck path $R=ENEN^{*}EEN^{*}EENNN\in\mathcal{R}_{8}$ may be decomposed as $R=P_{1}EP_{2}N$, with $P_{1}=ENEN^{*}\in\mathcal{P}_{3}$ and $P_{2}=EN^{*}EENN\in\mathcal{R}_{4}$.}\label{fig:carla_conj_proof2}
\end{figure}

Next we prove Theorem~\ref{carlaConjecture}, which gives a generating function for the sequence $\left|\bI_{n}\left(\underline{>,\leq}\right)\right|=\left|\bI_{n}\left(>,\leq,-\right)\right|$.

\begin{proof}[Proof of Theorem~\ref{carlaConjecture}]
First we describe a bijection $\varphi $ between $\bI_{n}\left(\underline{>,\leq}\right)$ (which equals $\bI_{n}\left(>,\leq,-\right)$ by Equation~\eqref{eq:Iclassical}) and $\mathcal{R}_{n}$.
Let $e\in\bI_{n}\left(\underline{>,\leq}\right)$. By Lemma~\ref{lem:A071356}, there exists $1\le j\le n$ such that $e_{1}\leq e_{2}\leq\dots\leq e_{j}>e_{j+1}>\dots >e_{n}$. Let $P$ be the corresponding underdiagonal lattice path from $(0,0)$ to the line $x=n$, using steps $N=(0,1)$, $S=(0,-1)$ and $E=(1,0)$, having $n$ steps $E$ at heights given by $e_1,\dots,e_n$.
We construct $\varphi(e)\in\mathcal{R}_{n}$ as follows; see Figure~\ref{fig:carla_conj_proof1}(a)(b) for an example.
\begin{enumerate}[label=\arabic*)]
\item For every $E$ step in the descending portion of $P$ (that is, to the right of the line $x=j$), which corresponds to an entry $e_i$ with $i>j$, mark the $N$ step in the ascending portion of $P$ going from height $e_i$ to height $e_i+1$, turning it into an $N^{*}$ step.
\item Erase the descending portion of $P$, and instead append $j-e_j$ $N$ steps. Let $\varphi(e)\in \mathcal{R}_n$ be the resulting path from the origin to $(j,j)$.
\end{enumerate}

\begin{figure}[htb]
\begin{centering}
\begin{tabular}{c@{\qquad}c@{\qquad}c}
\includegraphics[scale=0.48]{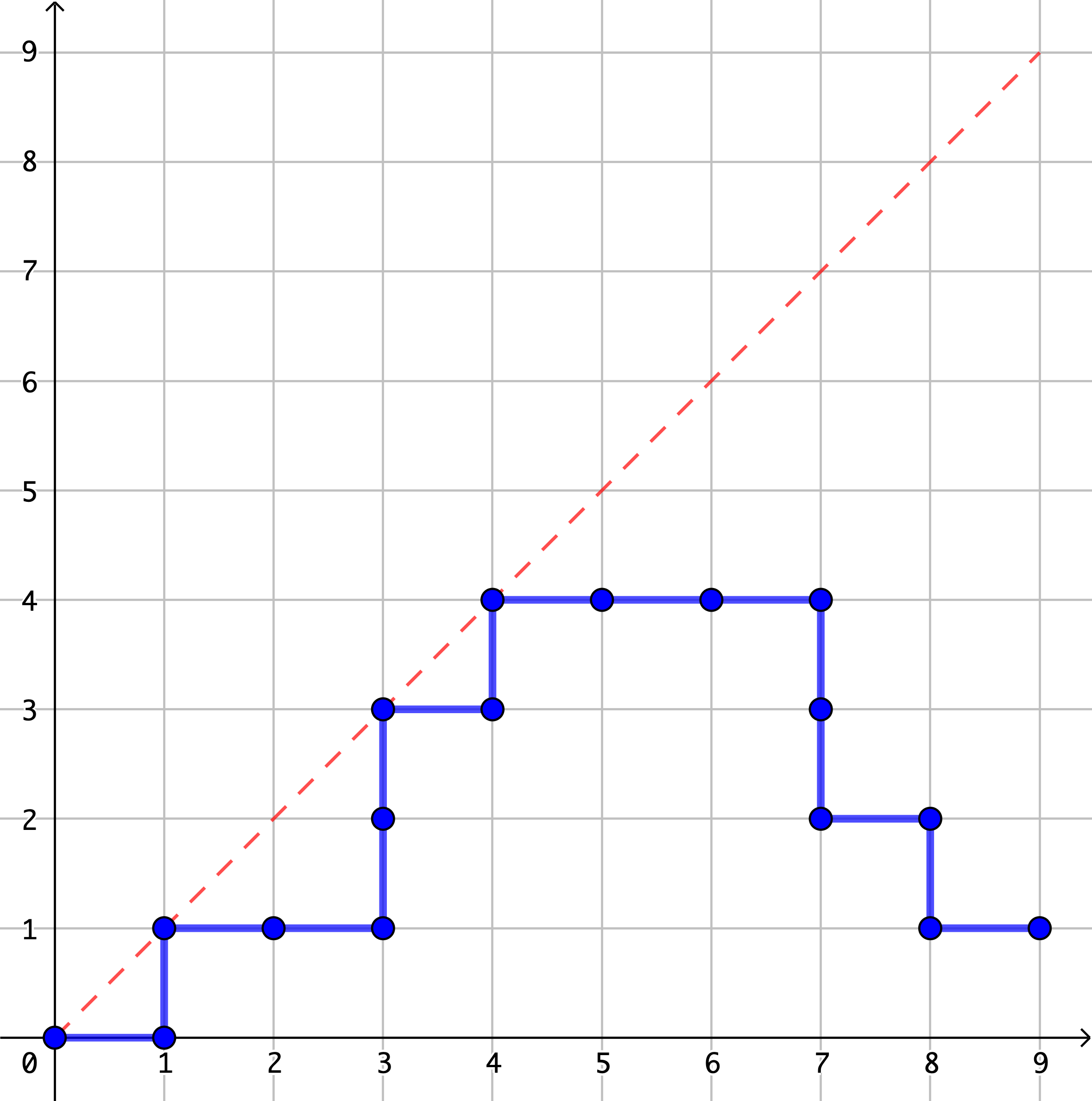}
&
\includegraphics[scale=0.48]{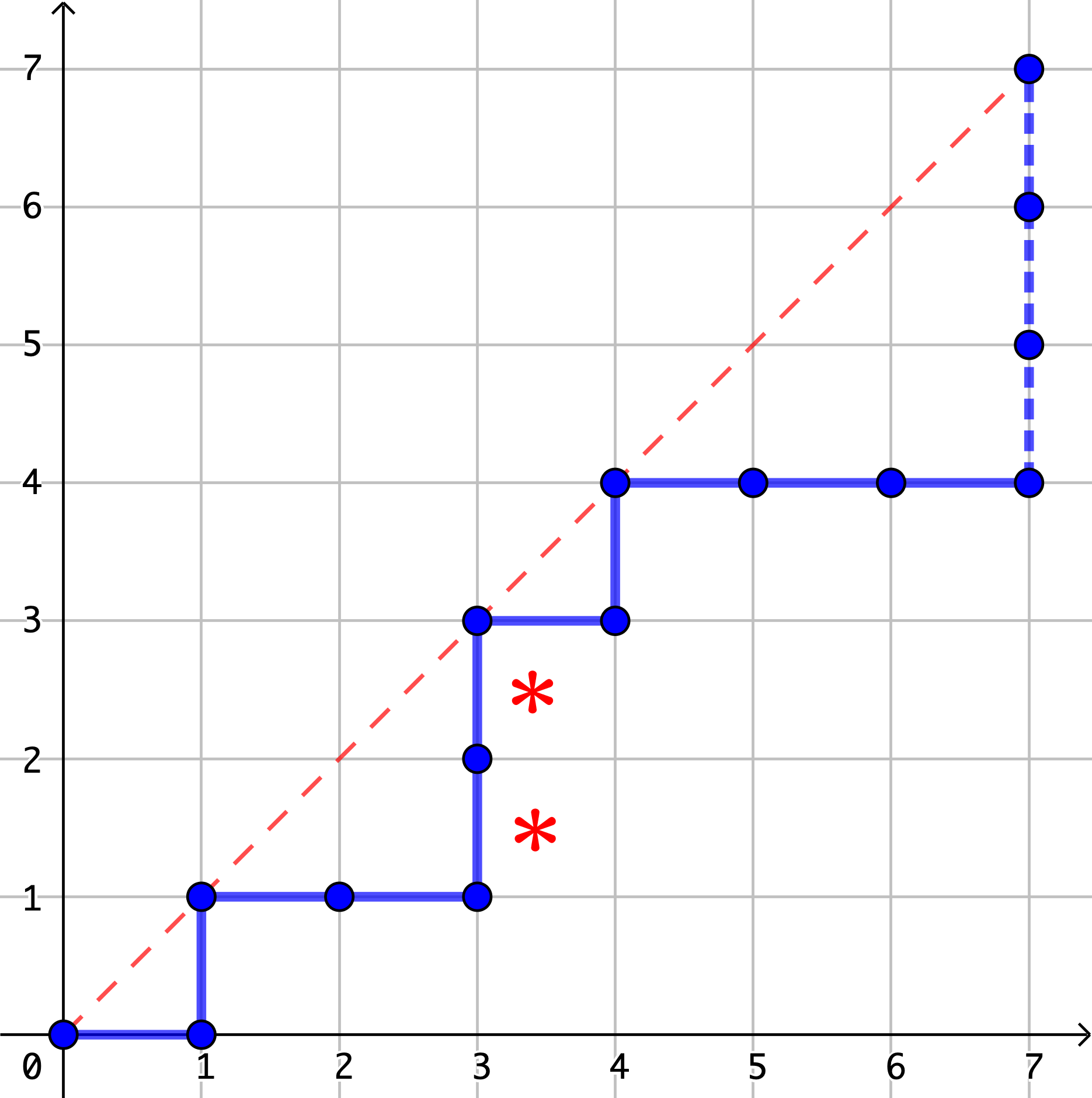}
&
\includegraphics[scale=0.48]{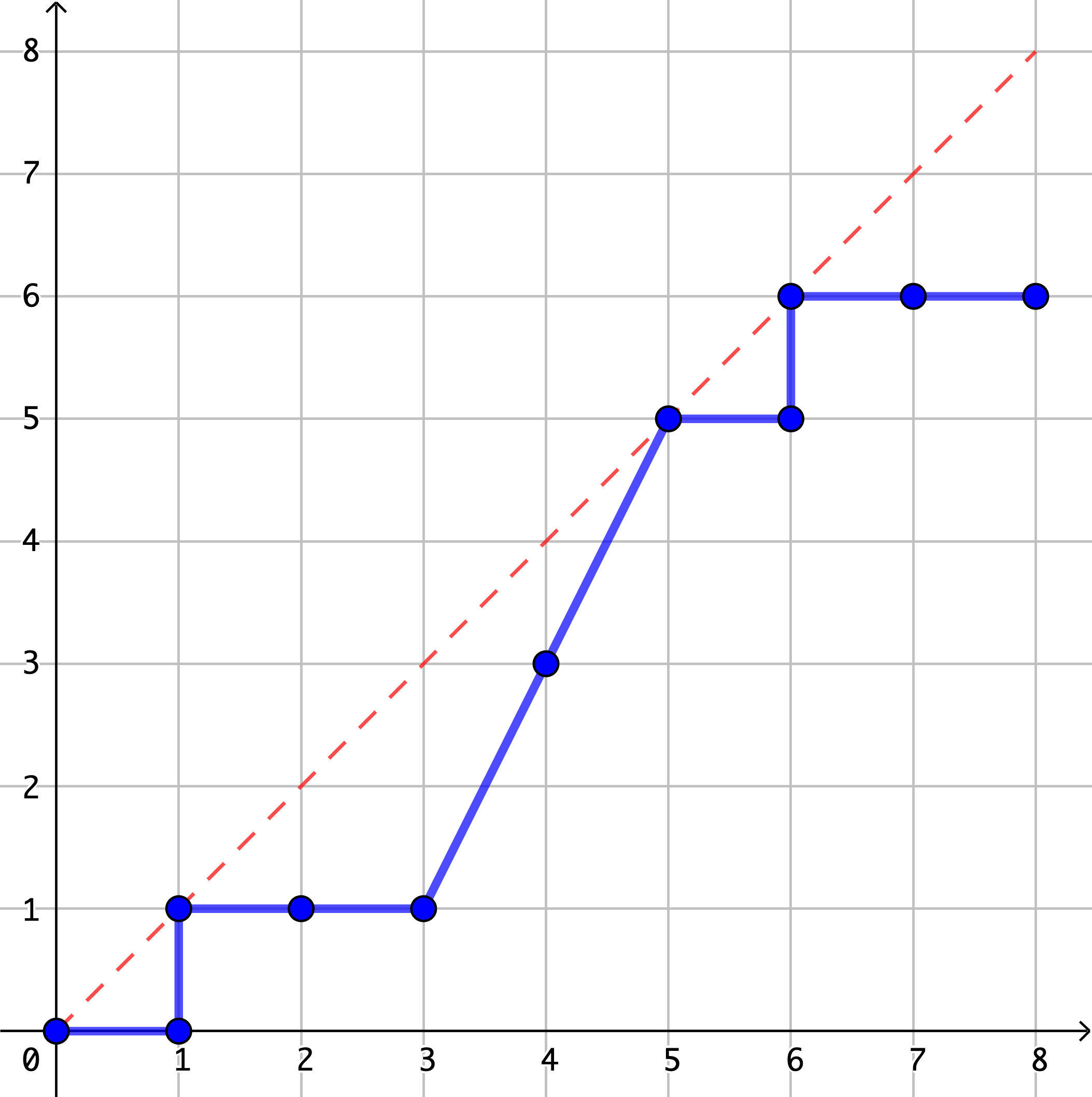}
\\
(a) & (b) & (c)
\end{tabular}
\end{centering}
\caption{(a) The inversion sequence $e=011344421\in\bI_{9}\left(\underline{>,\leq}\right)$ represented as the underdiagonal lattice path ${P=ENEENNENEEESSESE}$. (b) Its corresponding marked Dyck path ${\varphi(e)=ENEEN^{*}N^{*}ENEEENNN}\in\mathcal{R}_{9}$. (c) The path $R'=ENEEDDENEEE\in\mathcal{R}'_{8}$ corresponding to $e$ under the bijection $\varphi'$.}\label{fig:carla_conj_proof1}
\end{figure}

It is clear that $\varphi:\bI_{n}\left(\underline{>,\leq}\right)\to\mathcal{R}_{n}$ is a bijection, so it suffices to enumerate $\mathcal{R}_{n}$.

For $n\geq 1$, every $R\in\mathcal{R}_{n}$ can be decomposed uniquely as $R=P_{1}EP_{2}N$, where $P_{1}\in\mathcal{P}_{j}$ and $P_{2}\in\mathcal{R}_{n-j-1}$ for some $0\leq j\leq n-1$ (see Figure~\ref{fig:carla_conj_proof2} for an example). It follows that if we define $P(z)=\sum_{n\geq 0}\left|\mathcal{P}_{n}\right|z^{n}$ and $R(z)=\sum_{n\geq 0}\left|\mathcal{R}_{n}\right|z^{n}$, these generating functions are related by $R(z)=1+z\,P(z)R(z)$, and so
\[
R(z)=\frac{1}{1-z\, P(z)}.
\]

To find $P(z)$, think of marked Dyck paths as Dyck paths with two kinds of vertical steps: $N$ (contributing $1$ to the size) and $N^*$  (contributing $2$ to the size), whereas $E$ steps do not contribute to the size. If follows that $P(z)=C(z+z^{2})$, with $C(z)$ given by Equation~\eqref{eq:Catalan}. Indeed, the monomials $z$ and $z^{2}$ account for the choice of $N$ and $N^{*}$ steps, respectively. We deduce that
\[
R(z)=\frac{1}{1-z\, C(z+z^{2})}=\frac{2+2z}{1+2z+\sqrt{1-4z-4z^{2}}}=\frac{1 + 2z - \sqrt{1 - 4z - 4z^{2}}}{4z}.\qedhere
\]
\end{proof}

Martinez and Savage's original conjecture \cite[Section~2.19]{MartinezSavageII} stated that $\left|\bI_{n}\left(>,\leq,-\right)\right|$ coincides with sequence A071356 from~\cite{OEIS}, which enumerates the sets $\mathcal{R}'_{n}$ consisting of underdiagonal paths from $(0,0)$ to the line $x=n$ with steps $N$, $E$ and $D=(1,2)$. Even though Theorem~\ref{carlaConjecture} already proves their conjecture, we can also give a direct bijective proof, by exhibiting a bijection between  $\bI_{n}\left(>,\leq,-\right)$ and $\mathcal{R}'_{n-1}$. Indeed, given $e\in\bI_{n}\left(>,\leq,-\right)$, we can construct its corresponding marked Dyck path with an unmarked tail, $R=\varphi(e)\in\mathcal{R}_{n}$. Replacing each marked step $N^{*}$ in $R$ by a $D$ step and deleting the last run of $N$ steps as well as the $E$ step immediately preceding it, we obtain a path $R'\in\mathcal{R}'_{n-1}$ (see Figure~\ref{fig:carla_conj_proof1}(c) for an example). Defining $\varphi'(e)=R'$, it is straightforward to verify that $\varphi'$ is a bijection between $\bI_{n}\left(>,\leq,-\right)$ and $\mathcal{R}'_{n-1}$.

For $e\in\bI_n$, define the statistic $\dist(e)=\left|\left\{e_{1},e_{2},\ldots,e_{n}\right\}\right|$, that is, the number of distinct entries of $e$.
From numerical evidence, Martinez and Savage also conjecture in~\cite[Section~2.19]{MartinezSavageII} that the distribution of this statistic is symmetric on $\bI_{n}\left(>,\leq,-\right)$. Next we prove this conjecture by interpreting the statistic $\dist$ in terms of marked Dyck paths, and finding the corresponding bivariate generating function refining Theorem~\ref{carlaConjecture}.

\begin{thm}\label{thm:A071356}
\[
\sum_{n\geq 0}\sum_{e\in\bI_{n}\left(\underline{>,\leq}\right)}z^{n}t^{\dist(e)}=\frac{1+z(3-t)-\sqrt{1-z\left(2+2t-z+6zt-zt^{2}\right)}}{4z}.
\]
\end{thm}
\begin{proof}
Let $e\in\bI_n\left(\underline{>,\leq}\right)$, and let $R=\varphi(e)$, where $\varphi$ is the bijection in the proof of Theorem~\ref{carlaConjecture}. Define an {\em elbow} of a path in $\mathcal{P}_{n}$ to be an occurrence of a horizontal step $E$ immediately followed by a vertical step $N$ or $N^*$. Then $\dist(e)$ becomes the following statistic on $R$, which we also denote by $\dist$ with some abuse of notation:
\[
\dist(R)=\# \left\{\text{elbows in }R\right\}+\# \left\{N^*\text{ steps in } R\text{ that are not part of an elbow}\right\}.
\]
For instance, the marked Dyck path ${R=ENEEN^{*}N^{*}ENEEENNN}$ in Figure~\ref{fig:carla_conj_proof1}(b) has 4 elbows and 1 marked step that is not part of an elbow, so $\dist(R)=5$. Define the bivariate generating functions
\[
R(z,t)=\sum_{n\geq 0}\sum_{R\in\mathcal{R}_{n}}z^{n}t^{\dist(R)}=\sum_{n\geq 0}\sum_{e\in\bI_{n}\left(\underline{>,\leq}\right)}z^{n}t^{\dist(e)}
\quad\text{and}\quad
P(z,t)=\sum_{n\geq 0}\sum_{P\in\mathcal{P}_{n}}z^{n}t^{\dist(P)}.
\]

Decomposing $R\in\mathcal{R}_{n}$ for $n\ge1$ as $R=P_{1}EP_{2}N$, where $P_{1}\in\mathcal{P}_{j}$ and $P_{2}\in\mathcal{R}_{n-j-1}$ for some $0\leq j\le n-1$, and noting that the last two steps of $R$ form an elbow when $P_2$ is empty, we get the equation
\[
R(z,t)=1 + zt\, P(z,t) + z\, P(z,t)\left(R(z,t)-1\right),
\]
and so
\begin{equation}\label{eq:A071356_3}
R(z,t) = \frac{1-z(1-t)P(z,t)}{1-z\, P(z,t)}.
\end{equation}

To find $P(z,t)$, note that if a marked Dyck path $P$ is {\em irreducible} (i.e., it is nonempty and it returns to the diagonal only at the end), then $P=EP'N$ or $P=EP'N^{*}$, where $P'$ is a marked Dyck path. When $P'$ is empty, such paths $P$ contribute $zt+z^{2}t$ to the generating function, since they consist of an elbow; when $P'$ is nonempty, they contribute $\left(z+z^{2}t\right)\left(P(z,t)-1\right)$. Since every path in $\mathcal{P}_{n}$ can be decomposed uniquely as a sequence of irreducible paths, we deduce that
\[
P(z,t) = \frac{1}{1-\left[zt+z^{2}t+\left(z+z^{2}t\right)\left(P(z,t)-1\right)\right]}.
\]
Solving for $P(z,t)$ and taking the solution without negative powers in its series expansion, we find that
\[
P(z,t) = \frac{1+z(1-t)- \sqrt{\left(1+z(1-t)\right)^{2}-4\left(z+z^{2}t\right)}}{2\left(z+z^{2}t\right)}.
\]
It then follows from Equation~\eqref{eq:A071356_3} that
\[
R(z,t) = \frac{1+z(3-t)-\sqrt{1-z\left(2+2t-z+6zt-zt^{2}\right)}}{4z}.\qedhere
\]
\end{proof}

In Section~\ref{sec:unimodal} we will show how the argument in the proof of Theorem~\ref{thm:A071356} can be generalized to find bivariate generating functions for the number of inversion sequences $e$ satisfying certain unimodality conditions, with respect to the length of $e$ and $\dist(e)$.

Now we can finally show that the statistic $\dist$ has a symmetric distribution on $\bI_{n}\left(>,\leq,-\right)=\bI_n\left(\underline{>,\leq}\right)$, proving Martinez and Savage's conjecture.

\begin{cor}\label{cor:A071356}
For all $n\ge d\ge 1$,
$$|\{e\in\bI_{n}\left(\underline{>,\leq}\right): \dist(e)=d\}|=|\{e\in\bI_{n}\left(\underline{>,\leq}\right): \dist(e)=n+1-d\}|.$$
\end{cor}

\begin{proof}
Write $R(z,t)=\sum_{n\geq 0}U_{n}(t)z^{n}$, where $U_{n}(t)=\sum_{e\in\bI_{n}\left(\underline{>,\leq}\right)}t^{\dist(e)}=\sum_{i=1}^{n}u_{i,n}t^{i}$.
Then
\begin{align*}
t\, R\left(zt,\frac{1}{t}\right) &= t+t\,\sum_{n\geq 1}\left(u_{1,n}\frac{1}{t}+u_{2,n}\frac{1}{t^{2}}+\dots+ u_{n,n}\frac{1}{t^{n}}\right)(zt)^{n}\\
&= t + \sum_{n\geq 1}\left(u_{1,n}t^{n}+u_{2,n}t^{n-1}+\dots+u_{n,n}t\right)z^{n}.
\end{align*}
Thus, in order to prove that $u_{d,n}=u_{n+1-d,n}$ for $n\ge d\ge 1$, it is enough to show that $R(z,t)-1=t\,R\left(zt,\frac{1}{t}\right)-t$. This is now immediate from Theorem~\ref{thm:A071356}, since
\[
t\, R\left(zt,\frac{1}{t}\right)-t = \frac{1-z(1+t)-\sqrt{1-z\left(2+2t-z+6zt-zt^{2}\right)}}{4z}= R(z,t)-1.\qedhere
\]
\end{proof}

A different proof of this corollary, showing $\gamma$-positivity of $U_n(t)$, has recently been given by Cao, Jin and Lin~\cite[Theorem~5.1]{CaoJinLin}.

\subsection{The pattern $\left(\underline{=,\neq}\right)$}\label{subsec:A003422}
The characterization of inversion sequences avoiding this pattern is relatively simple.

\begin{lem}\label{lem:A003422}
For $n\ge1$,
$$\bI_{n}\left(\underline{=,\neq}\right)=\{e\in\bI_{n}:
e_{1}\neq e_{2}\neq\dots\neq e_{j}=e_{j+1}=\dots =e_{n}\text{ for some }1\leq j\leq n\}.$$
\end{lem}

\begin{proof}
Given $e\in\bI_{n}\left(\underline{=,\neq}\right)$, let $1\leq j\leq n-1$ be the smallest index such that $e_{j}=e_{j+1}$, or let $j=n$ if there is not such index. Since $e$ avoids $\left(\underline{=,\neq}\right)$, we must have $e_{j}=e_{j+1}=\dots =e_{n}$.
\end{proof}

Next we use this characterization to find a formula for $\left|\bI_{n}\left(\underline{=,\neq}\right)\right|$.
The corresponding sequence, which is listed as A003422 in~\cite{OEIS}, is sometimes referred to as the {\em left factorial} of~$n$.

\begin{prop} \label{prop:A003422}
For $n\geq 1$, $\left|\bI_{n}\left(\underline{=,\neq}\right)\right|=0!+1!+2!+\dots+(n-1)!$.
\end{prop}
\begin{proof}
By Lemma~\ref{lem:A003422}, we can express $\bI_{n}\left(\underline{=,\neq}\right)$ as a disjoint union $\bigsqcup_{j=1}^{n}A_{n,j}$, where $A_{n,j}=\left\{e\in\bI_{n}:e_{1}\neq e_{2}\neq\dots\neq e_{j}=e_{j+1}=\dots =e_{n}\right\}$.

To specify an element $e\in A_{n,j}$, after setting $e_{1}=0$, we have $i-1$ choices for $e_{i}$
for each $2\leq i\leq j$,
because $e_{i}\in\left\{0,1,\ldots i-1\right\}\backslash \left\{e_{i-1}\right\}$. Once $e_{j}$ is chosen, the entries $e_{j+1},e_{j+2},\ldots e_{n}$ are forced. It follows that $\left|A_{n,j}\right|=(j-1)!$, and consequently,
\[
\left|\bI_{n}\left(\underline{=,\neq}\right)\right|=\sum_{j=1}^{n}\left|A_{n,j}\right|=0!+1!+2!+\dots+(n-1)!.\qedhere
\]
\end{proof}

\subsection{The patterns  $\left(\underline{\geq,\geq}\right)\stackrel{ss}{\sim}\left(\underline{<,<}\right)$}\label{subsec:A049774}

An occurrence of $\left(\underline{<,<}\right)$ in an inversion sequence is precisely an occurrence of the consecutive pattern $\underline{012}$, so $\bI_{n}\left(\underline{<,<}\right)=\bI_{n}\left(\underline{012}\right)$.
It was shown in \cite[Proposition~3.19]{AuliElizalde} that if $\Theta:S_{n}\rightarrow \bI_{n}$ is the map defined by~\eqref{eq:theta_bijection} and $\pi\in S_{n}$, then $\pi$ avoids the consecutive permutation pattern $\underline{321}$ if and only $\Theta(\pi)$ avoids $\underline{012}$. Using Theorem~\ref{EquivIneq}(ii) and~\cite[Theorem 4.1]{ElizaldeNoy}, the next corollary follows.

\begin{cor}[\cite{AuliElizalde,ElizaldeNoy}]\label{cor:A049774} Let $\pi\in S_{n}$, and let $e=\Theta(\pi)$ be its corresponding inversion sequence. Then $\pi$ avoids
$\underline{321}$ if and only $e$ avoids $\left(\underline{<,<}\right)$. In particular,
$\left|\bI_{n}\left(\underline{\geq,\geq}\right)\right|=\left|\bI_{n}\left(\underline{<,<}\right)\right|=\left|\bI_{n}\left(\underline{012}\right)\right|=\left|S_{n}\left(\underline{321}\right)\right|$ for all $n$, and the corresponding exponential generating function is
$$\sum_{n\ge0} \left|\bI_{n}\left(\underline{<,<}\right)\right| \frac{z^n}{n!}=\frac{\sqrt{3}}{2}\frac{\exp\left(z/2\right)}{\cos\left(\pi/6+\sqrt{3}z/2\right)}.$$
\end{cor}

\subsection{The pattern $\left(\underline{\neq,=}\right)$}\label{subsec:A000522}
The techniques in this section are similar to those in Section~\ref{subsec:A003422}.

\begin{lem}\label{lem:A000522}
For $n\geq 1$,
$$\bI_{n}\left(\underline{\neq,=}\right)=\{e\in\bI_{n}:
e_{1}= e_{2}=\dots =e_{j}\neq e_{j+1}\neq\dots \neq e_{n}\text{ for some }1\leq j\leq n\}.$$
\end{lem}

\begin{proof}
Given $e\in\bI_{n}\left(\underline{\neq,=}\right)$, let $1\leq j\leq n-1$ be the smallest index such that $e_{j}\neq e_{j+1}$, or let $j=n$ if there is no such index.
Since $e$ avoids $\left(\underline{\neq,=}\right)$, we have $e_{j}\neq e_{j+1}\neq\ldots\neq e_{n}$.
\end{proof}

We can use the description of inversion sequences avoiding the pattern $\left(\underline{\neq,=}\right)$ given by Lemma~\ref{lem:A000522} to enumerate them. The corresponding sequence is listed as A000522 in~\cite{OEIS}.

\begin{prop}\label{prop:A000522}
For $n\geq 1$, $\left|\bI_{n}\left(\underline{=,\neq}\right)\right|=\sum_{i=0}^{n-1}(n-1)!/i!$.
\end{prop}

\begin{proof}
By Lemma~\ref{lem:A000522}, we can express $\bI_{n}\left(\underline{\neq,=}\right)$ as a disjoint union $\bigsqcup_{j=1}^{n}A_{n,j}$, where $A_{n,j}=\left\{e\in\bI_{n}:e_{1}= e_{2}=\dots =e_{j}\neq e_{j+1}\neq\dots \neq e_{n}\right\}$.

To specify an element $e\in A_{n,j}$, after setting $e_{1}=\dots=e_j=0$, we have $i-1$ choices for $e_{i}$
for each $j+1\leq i\leq n$, because $e_{i}\in\left\{0,1,\ldots i-1\right\}\backslash \left\{e_{i-1}\right\}$.
Thus, $\left|A_{n,j}\right|=j(j+1)\dots(n-1)=(n-1)!/(j-1)!$, and
\[
\left|\bI_{n}\left(\underline{\neq,=}\right)\right|=\sum_{j=1}^{n}\left|A_{n,j}\right|
=\sum_{i=0}^{n-1}\frac{(n-1)!}{i!}.\qedhere
\]
\end{proof}

\subsection{The patterns $\left(\underline{\geq,>}\right)\stackrel{ss}{\sim}\left(\underline{>,\geq}\right)$}\label{subsec:A200403}
We showed in Proposition~\ref{prop:A200403} that
$\left|\bI_{n}\left(\underline{>,\geq}\right)\right|=\left|\bI_{n}\left(\underline{\geq,>}\right)\right|=\left|S_{n}\left(\underline{124}3\right)\right|$. This sequence appears as A200403 in~\cite{OEIS}, but no closed formula for it is known.

We can obtain a recurrence to compute $\left|\bI_{n}\left(\underline{>,\geq}\right)\right|$ by introducing some refinements.
Let $\bI_{n,k}=\left\{e\in\bI_{n}:e_{n}=k\right\}$, and define $\bI_{n,k}\left(\underline{R_{1},R_{2}}\right)=\bI_{n,k}\cap \bI_{n}\left(\underline{R_{1},R_{2}}\right)$ for any consecutive pattern of relations $\left(\underline{R_{1},R_{2}}\right)$. Note that  $\bI_{n}\left(\underline{R_{1},R_{2}}\right)=\bigsqcup_{k=0}^{n-1}\bI_{n,k}\left(\underline{R_{1},R_{2}}\right)$, and that $\bI_{n,k}\left(\underline{R_{1},R_{2}}\right)=\emptyset$ unless $0\le k<n$.

Proposition~\ref{prop:A200403_biject} shows that $\left|\bI_{n,k}\left(\underline{>,\geq}\right)\right|=\left|\bI_{n,k}\left(\underline{\geq,>}\right)\right|$. To compute these numbers, we define
\[
\bI_{n,k}^>\left(\underline{>,\geq}\right)=\left\{e\in\bI_{n,k}\left(\underline{>,\geq}\right):e_{n-1}>e_{n}\right\}.
\]

\begin{prop}\label{prop:A200403_2} For $n\geq 2$, the sequences $\left|\bI_{n,k}\left(\underline{>,\geq}\right)\right|$ and $\left|\bI_{n,k}^{>}\left(\underline{>,\geq}\right)\right|$ satisfy the recurrences
\begin{equation}\label{eq:A200403_1}
\left|\bI_{n,k}\left(\underline{>,\geq}\right)\right| = \begin{cases}
      \left|\bI_{n-1}\left(\underline{>,\geq}\right)\right| & \text{if }k=n-1,\\
      \left|\bI_{n,k+1}\left(\underline{>,\geq}\right)\right|-\left|\bI_{n-1,k}^>\left(\underline{>,\geq}\right)\right| & \text{if }0\leq k\leq n-2,\\
      0 & \text{otherwise,}
      \end{cases}
\end{equation}
\begin{equation}\label{eq:A200403_2}
\left|\bI_{n,k}^{>}\left(\underline{>,\geq}\right)\right| =\begin{cases}
\left|\bI_{n,k+1}^{>}\left(\underline{>,\geq}\right)\right|-\left|\bI_{n-1,k+1}^{>}\left(\underline{>,\geq}\right)\right|+\left|\bI_{n-1,k+1}\left(\underline{>,\geq}\right)\right| & \text{if }0\leq k\leq n-2,\\
      0 & \text{otherwise,}
\end{cases}
\end{equation}
with initial conditions $\left|\bI_{1}\left(\underline{>,\geq}\right)\right|=\left|\bI_{1,0}\left(\underline{>,\geq}\right)\right|=1$, $\left|\bI_{1,k}\left(\underline{>,\geq}\right)\right|=0$ for $k>0$, and $\left|\bI_{1,k}^{>}\left(\underline{>,\geq}\right)\right|=0$ for all $k$.
\end{prop}

Equations~\eqref{eq:A200403_1} and~\eqref{eq:A200403_2} allow us to compute the values $\left|\bI_{n,k}\left(\underline{>,\geq}\right)\right|$ and $\left|\bI_{n,k}^{>}\left(\underline{>,\geq}\right)\right|$ recursively, since an entry indexed by $(n,k)$ depends only on entries indexed by $(n,j)$ where $k<j<n$, or by $(n-1,j)$ for some~$j$.

\begin{proof}
Note that $e\in\bI_{n,k}\left(\underline{>,\geq}\right)$ if and only if $e_n=k$ and
\[{e_{1}e_{2}\ldots e_{n-1}}\in\bI_{n-1}\left(\underline{>,\geq}\right)\backslash\bigsqcup_{j=k}^{n-2}\bI_{n-1,j}^{>}\left(\underline{>,\geq}\right),
\]
where the subtracted term guarantees that $e_{n-2}e_{n-1}e_n$ is not an occurrence of $\left(\underline{>,\geq}\right)$.
Thus,
\begin{equation}\label{eq:A200403_3}
\left|\bI_{n,k}\left(\underline{>,\geq}\right)\right|=\left|\bI_{n-1}\left(\underline{>,\geq}\right)\right|-\sum_{j=k}^{n-2}\left|\bI_{n-1,j}^{>}\left(\underline{>,\geq}\right)\right|.
\end{equation}
In particular, this implies that $\left|\bI_{n,n-1}\left(\underline{>,\geq}\right)\right|=\left|\bI_{n-1}\left(\underline{>,\geq}\right)\right|$. If $0\leq k\leq n-2$, then subtracting Equation~\eqref{eq:A200403_3} from the same equation with $k+1$ substituted for $k$, we may write
$$ \left|\bI_{n,k+1}\left(\underline{>,\geq}\right)\right| - \left|\bI_{n,k}\left(\underline{>,\geq}\right)\right| =
\left|\bI_{n-1,k}^{>}\left(\underline{>,\geq}\right)\right|.$$
This proves the recurrence~\eqref{eq:A200403_1}.

Now, $e\in\bI_{n,k}^{>}\left(\underline{>,\geq}\right)$ if and only if $e_{n}=k$ and
${e_{1}e_{2}\ldots e_{n-1}}\in\bigsqcup_{j=k+1}^{n-2}\bI_{n-1,j}\left(\underline{>,\geq}\right)\backslash\bI_{n-1,j}^{>}\left(\underline{>,\geq}\right)$. Hence,
\begin{equation}\label{eq:A200403_4}
\left|\bI_{n,k}^{>}\left(\underline{>,\geq}\right)\right| = \sum_{j=k+1}^{n-2} \left(\left|\bI_{n-1,j}\left(\underline{>,\geq}\right)\right| - \left|\bI_{n-1,j}^{>}\left(\underline{>,\geq}\right)\right|\right).
\end{equation}
In particular, $\left|\bI_{n,n-1}^{>}\left(\underline{>,\geq}\right)\right|=0$. If $0\leq k\leq n-2$, then subtracting Equation~\eqref{eq:A200403_4} from the same equation with $k+1$ substituted for $k$, we may write
$$
\left|\bI_{n,k+1}^{>}\left(\underline{>,\geq}\right)\right| - \left|\bI_{n,k}^{>}\left(\underline{>,\geq}\right)\right| = - \left|\bI_{n-1,k+1}\left(\underline{>,\geq}\right)\right| + \left|\bI_{n-1,k+1}^{>}\left(\underline{>,\geq}\right)\right|.
$$
This proves Equation~\eqref{eq:A200403_2}.
\end{proof}

More generally, for any $R_{1},R_{2}\in\left\{\leq,\geq,<,>,=,\neq\right\}$, we define the refinement
\begin{equation}\label{eq:refinement}
\bI_{n,k}^{R_{1}}\left(\underline{R_{1},R_{2}}\right)=\left\{e\in\bI_{n,k}\left(\underline{R_{1},R_{2}}\right):e_{n-1}R_1e_{n}\right\}.
\end{equation}
Then, arguing as in the proof of Proposition~\ref{prop:A200403_2}, we can write
\[
\left|\bI_{n,k}\left(\underline{R_{1},R_{2}}\right)\right| = \left|\bI_{n-1}\left(\underline{R_{1},R_{2}}\right)\right| - \sum_{\substack{j=0 \\ jR_{2}k}}^{n-2}\bI_{n-1,j}^{R_{1}}\left(\underline{R_{1},R_{2}}\right)
\]
and
\[
\left|\bI_{n,k}^{R_{1}}\left(\underline{R_{1},R_{2}}\right)\right| = \sum_{\substack{j=0 \\ jR_{1}k}}^{n-2}\left|\bI_{n-1,j}\left(\underline{R_{1},R_{2}}\right)\right| - \sum_{\substack{j=0 \\ jR_{1}k,\, jR_{2}k}}^{n-2}\bI_{n-1,j}^{R_{1}}\left(\underline{R_{1},R_{2}}\right).
\]
These two equations can be simplified, depending on $R_{1}$ and $R_{2}$, to obtain recurrences analogous to Equations~\eqref{eq:A200403_1} and~\eqref{eq:A200403_2}, which allow us to compute the values $\left|\bI_{n,k}\left(\underline{R_{1},R_{2}}\right)\right|$ and $\left|\bI_{n,k}^{R_{1}}\left(\underline{R_{1},R_{2}}\right)\right|$ recursively. In particular, we can obtain such recurrences for each of the 12 Wilf equivalence classes of consecutive patterns of relations mentioned in the beginning of Section~\ref{sec:enumerative_results}, for which $\left|\bI_{n}\left(\underline{R_{1},R_{2}}\right)\right|$ does not match any sequence in~\cite{OEIS}.

We remark that, even though $\left|\bI_{n,k}\left(\underline{>,\geq}\right)\right|=\left|\bI_{n,k}\left(\underline{\geq,>}\right)\right|$ by
Proposition~\ref{prop:A200403_biject}, the resulting recurrences involving the refinements in~\eqref{eq:refinement} are different. In particular, while the quantities $\left|\bI_{n,k}^{>}\left(\underline{>,\geq}\right)\right|$ and $\left|\bI_{n,k}^{\geq}\left(\underline{\geq,>}\right)\right|$ are not equal in general, one can show that
$
\left|\bI_{n,k}^{>}\left(\underline{>,\geq}\right)\right|=\left|\bI_{n,k+1}^{\geq}\left(\underline{\geq,>}\right)\right|
$
for all $n$ and $k$.

\subsection{The pattern $\left(\underline{=,=}\right)$}\label{subsec:A052169}

An occurrence of $\left(\underline{=,=}\right)$ in an inversion sequence is precisely an occurrence of the consecutive pattern $\underline{000}$, so $\bI_{n}\left(\underline{=,=}\right)=\bI_{n}\left(\underline{000}\right)$. The next result rephrases
\cite[Proposition~3.1 and Corollary~3.3]{AuliElizalde} about inversion sequences that avoid $\underline{000}$. We use $d_n$ to denote  the number of derangements of length $n$, i.e., the number of permutations in $S_n$ with no fixed points.

\begin{prop}[\cite{AuliElizalde}]\label{prop:A052169}
The sequence $\left|\bI_{n}\left(\underline{=,=}\right)\right|$ satisfies the recurrence
\[
\left|\bI_{n}\left(\underline{=,=}\right)\right|=(n-1)\left|\bI_{n-1}\left(\underline{=,=}\right)\right|+(n-2)\left|\bI_{n-2}\left(\underline{=,=}\right)\right|
\]
for $n\geq 3$, with initial terms $\left|\bI_{1}\left(\underline{=,=}\right)\right|=1$ and $\left|\bI_{2}\left(\underline{=,=}\right)\right|=2$. It follows that
$$
\left|\bI_{n}\left(\underline{=,=}\right)\right|=\frac{(n+1)!-d_{n+1}}{n}.
$$
\end{prop}

\section{Inversion Sequences Satisfying Unimodality Conditions}\label{sec:unimodal}

In Lemma~\ref{lem:A071356} we enumerated inversion sequences satisfying a unimodality condition, which,
as observed by Martinez and Savage~\cite[Section~2.19]{MartinezSavageII}, are precisely those inversion sequences avoiding the pattern $\left(>,\leq,-\right)$.
This is one of 10 triples of binary relations of the form $\left(R_{1},R_{2},R_{3}\right)$ such that $\bI_{n}\left(R_{1},R_{2},R_{3}\right)$ is characterized by some type of unimodality condition. The enumeration of the inversion sequences in each case  was carried out in~\cite{MartinezSavageII}, with the results summarized in Table~\ref{tab:unimodal}.

\begin{table}[htbp]
\footnotesize
\begin{center}
 \begin{tabular}{ c@{\hskip 0.2in}l@{\hskip 0.2in}p{4cm} }
 \hline
 Pattern $\left(R_{1},R_{2},R_{3}\right)$ & unimodality condition for $e\in\bI_{n}\left(R_{1},R_{2},R_{3}\right)$ & $\left|\bI_{n}\left(R_{1},R_{2},R_{3}\right)\right|$ counted by
 \Tstrut\Bstrut\\
 \hline
 $\left(<,-,<\right)$ & $\exists j:e_{1}=\dots =e_{j}\leq e_{j+1}\geq 0=\dots =0$ & $1+\binom{n}{2}$ \Tstrut\\
 $\left(\neq,<,-\right)$ & $\exists j:e_{1}=\dots =e_{j}\leq e_{j+1}\geq e_{j+2}\geq\dots \geq e_{n}$ & $2^{n}-n$ \Tstrut\\
 $\left(\neq,\leq,-\right)$ & $\exists j:e_{1}=\dots =e_{j}< e_{j+1}> e_{j+2}>\dots >e_{n}$ & $F_{n+2}-1$ \Tstrut\\
 $\left(>,<,-\right)$ & $\exists j:e_{1}\leq\dots\leq e_{j}> e_{j+1}\geq e_{j+2}\geq\dots\geq e_{n}$ \Tstrut & OGF is
 $\frac{1+z-\sqrt{1-6z+5z^2}}{2z(2-z)}$
 \\
 $\left(>,\leq,-\right)$ & $\exists j:e_{1}\leq\dots\leq e_{j}> e_{j+1}> e_{j+2}>\dots >e_{n}$ \Tstrut & OGF is $\frac{1+2z-\sqrt{1-4z-4z^2}}{4z}$ \Tstrut\\
 $\left(>,\neq,-\right)$ & $\exists j:e_{1}\leq\dots\leq e_{j}\geq e_{j+1}= e_{j+2}=\dots =e_{n}$ & $1+\sum_{i=1}^{n-1}\binom{2i}{i-1}$ \Tstrut\\
 $\left(\geq,\neq,-\right)$ & $\exists j:e_{1}<\dots <e_{j}\geq e_{j+1}= e_{j+2}=\dots =e_{n}$ & $1+\binom{n}{2}$ \Tstrut\\
 $\left(=,<,-\right)$ & $\exists j:e_{1}<\dots <e_{j}\geq e_{j+1}\geq e_{j+2}\geq\dots\geq e_{n}$ & $2^{n-1}$ \Tstrut\\
 $\left(=,\leq,-\right)$ & $\exists j:e_{1}<\dots <e_{j}\geq e_{j+1}> e_{j+2}>\dots >e_{n}$ \Tstrut & $F_{n+1}$ \\
 $\left(\geq,\leq,\neq\right)$ & $\exists (j\leq i):e_{1}<\dots <e_{j}=\dots =e_{i}>\dots >e_{n}$ \Tstrut & $F_{n+2}-1$ \Tstrut\Bstrut\\
 \hline
 \end{tabular}
 \end{center}
 \caption{Patterns $\left(R_{1},R_{2},R_{3}\right)$ for which $\bI_{n}\left(R_{1},R_{2},R_{3}\right)$ is characterized by a unimodality condition. These enumerative results are due to Martinez and Savage~\cite{MartinezSavageII}. Patterns are listed in the same order from~\cite{MartinezSavageII}, rather than from our usual order (least avoided to most avoided).}\label{tab:unimodal}
 \vspace{-.7cm}
\end{table}

Martinez and Savage~\cite{MartinezSavageII} obtain these via {\it ad hoc} methods. In this section, we provide a unified approach, generalizing the method that we used in Section~\ref{subsec:A071356} to prove Theorem~\ref{carlaConjecture}, in order to recover their results. Furthermore, we find bivariate generating functions keeping track of the statistic $\dist$ (recording the number of distinct entries), in analogy to Theorem~\ref{thm:A071356}, for each of the patterns in Table~\ref{tab:unimodal}. For a triple of relations $\left(R_{1},R_{2},R_{3}\right)$, define
$$I_{\left(R_{1},R_{2},R_{3}\right)}(z,t)=\sum_{n\geq 0}\sum_{e\in\bI_{n}\left(R_{1},R_{2},R_{3}\right)}z^{n}t^{\dist(e)}.$$

We begin by considering the pattern $\left(>,<,-\right)$, and noting that $\bI_{n}\left(>,\leq,-\right)\subseteq\bI_{n}\left(>,<,-\right)$.
We generalize Definition~\ref{def:marked} by allowing different kinds of marked steps in a Dyck path.

\begin{defin} A {\em multi-marked Dyck path} $P$ is an underdiagonal path from $(0,0)$ to some point in the diagonal, with horizontal steps $E=(1,0)$ and infinitely many possible kinds of vertical steps $(0,1)$, denoted by $N$ and $N^{*}_t$ for $t\geq 2$ (the latter are called {\em marked} steps). Denoting by $E(P)$, $N(P)$ and $N^*_t(P)$ the number of $E$, $N$ and $N^*_t$ steps in $P$, respectively, the size of $P$ is defined as
\[
E(P)+\sum_{t\geq 2}(t-1)N^*_t(P)=N(P)+\sum_{t\geq 2}t\,N^*_t(P).
\]
We denote by $\mathcal{\widetilde{P}}_{n}$ the set of multi-marked Dyck paths of size $n$.

If $P\in\mathcal{\widetilde{P}}_{n}$ has at least one $N^*_t$ step in the last run of vertical steps, then we say that $P$ has a {\em marked tail}. Otherwise, we say that $P$ has an {\em unmarked tail}. We denote by $\mathcal{\widetilde{R}}_{n}$ the set of paths in $\mathcal{\widetilde{P}}_{n}$ with an unmarked tail.
\end{defin}

Multi-marked Dyck paths generalize marked Dyck paths from Definition~\ref{def:marked}, in the sense that any marked Dyck path in $\mathcal{P}_{n}$, after replacing $N^*$ steps with $N^*_2$ steps, can be seen as a multi-marked Dyck path in $\mathcal{\widetilde{P}}_{n}$ containing no steps $N^*_t$ with $t\ge3$. Hence, $\mathcal{P}_{n}\subseteq\mathcal{\widetilde{P}}_{n}$ and $\mathcal{R}_{n}\subseteq\mathcal{\widetilde{R}}_{n}$.

\begin{exa} The multi-marked Dyck path $R=ENEEN^*_4N^*_2ENEN^*_3EN$ has size 12, so $R\in\mathcal{\widetilde{P}}_{12}$. In fact, $R$ has an unmarked tail, so $R\in\mathcal{\widetilde{R}}_{12}$. This path is drawn on the right of Figure~\ref{fig:carla_conj_proof3}(b).
\end{exa}

\begin{figure}[htb]
\begin{centering}
\begin{tabular}{c@{\qquad}c}
\includegraphics[scale=0.6]{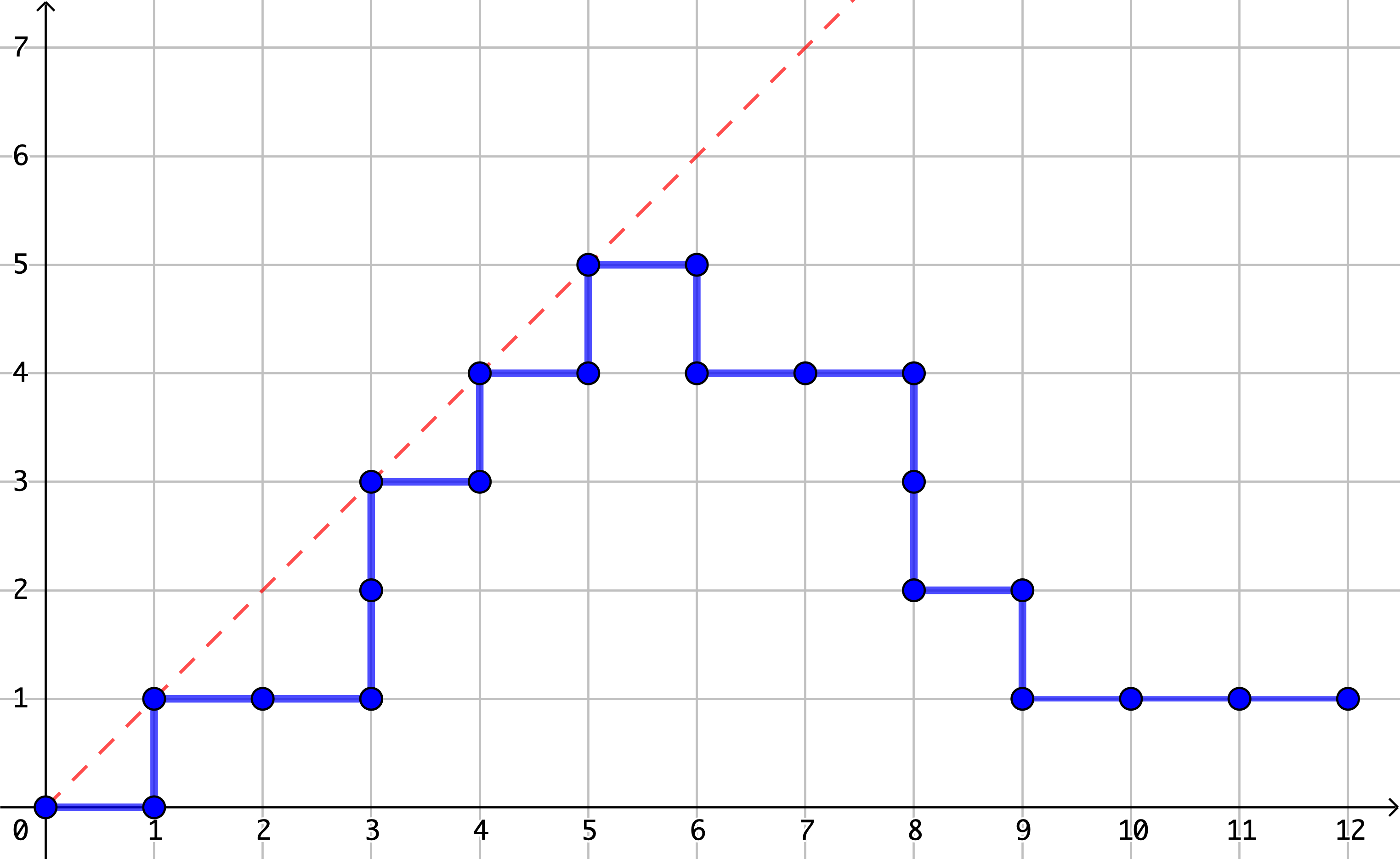}
&
\includegraphics[scale=0.6]{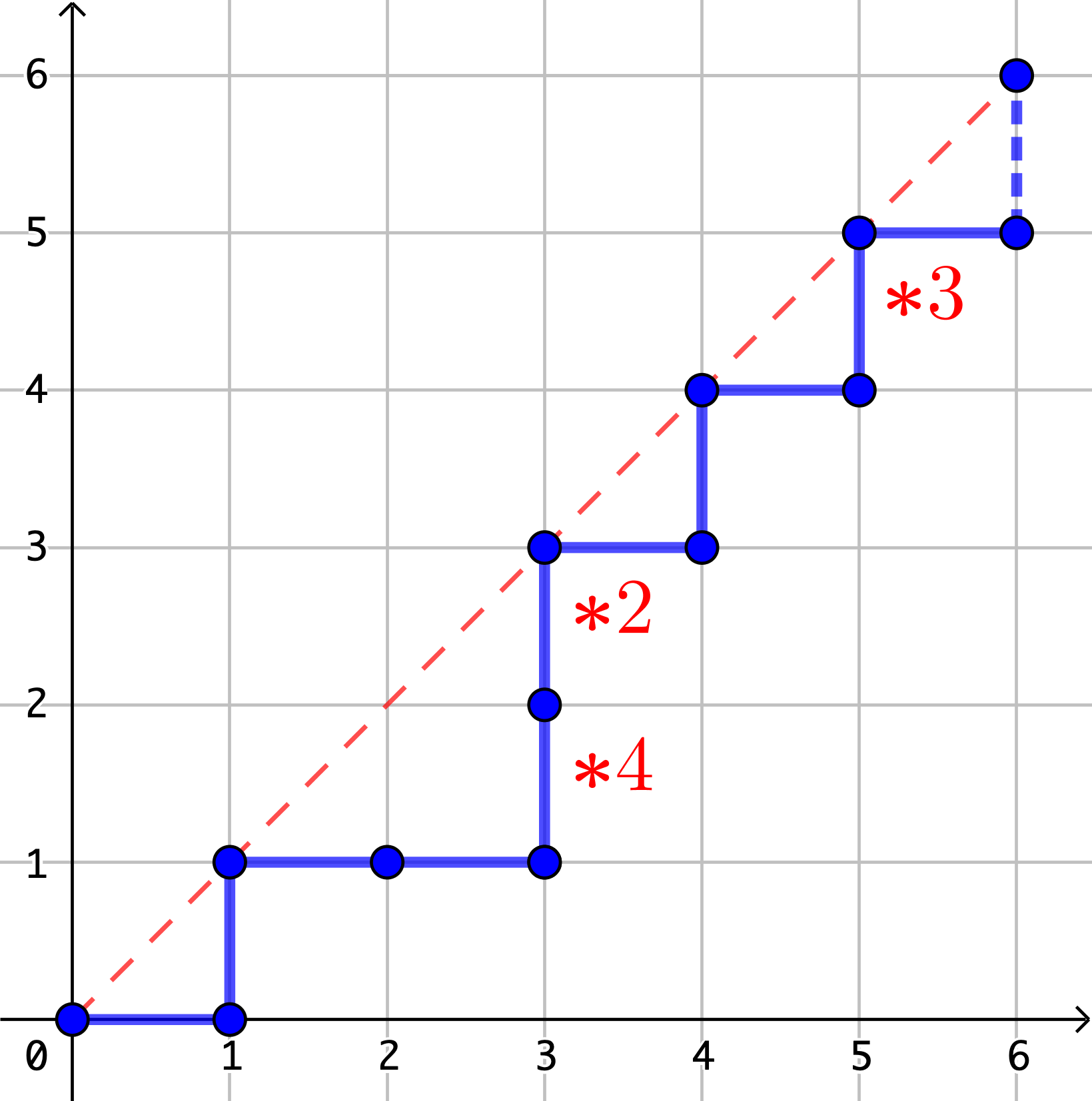}
\\
(a) & (b)
\end{tabular}
\end{centering}
\caption{(a) The inversion sequence $e=011345442111\in\bI_{12}\left(>,<,-\right)$ represented as an the underdiagonal lattice path $P=ENEENNENENESEESSESEEE$. (b) Its corresponding multi-marked Dyck path ${R=ENEEN^*_4N^*_2ENEN^*_3EN}\in\mathcal{\widetilde{R}}_{12}$.}\label{fig:carla_conj_proof3}
\end{figure}

\begin{thm}\label{thm:unimodal_1}
\[
I_{\left(>,<,-\right)}(z,t) = \frac{1+z(2-t)-z^{2}(1-t)-(1-z)\sqrt{(1+z-zt)^{2}-\frac{4z(1-z+zt)}{1-z}}}{2z(2-z)}.
\]
\end{thm}

\begin{proof}
We begin by describing a bijection $\varphi:\bI_{n}\left(>,<,-\right)\rightarrow\mathcal{\widetilde{R}}_{n}$, generalizing the bijection in the proof of Theorem~\ref{carlaConjecture}, which was between $\bI_{n}\left(>,\leq,-\right)$ and $\mathcal{R}_{n}$.

Let $e\in\bI_{n}\left(>,<,-\right)$, and let $j$ be such that $e_{1}\leq e_{2}\leq\cdots\leq e_{j}>e_{j+1}\geq e_{j+2}\geq\cdots\geq e_{n}$.
Let $P$ be the corresponding underdiagonal lattice path from $(0,0)$ to the line $x=n$, with $E$ steps at heights $e_{1},\ldots,e_{n}$, and the necessary $N$ and $S$ steps in between.
We construct $\varphi(e)\in\mathcal{\widetilde{R}}_{n}$ as follows; see Figure~\ref{fig:carla_conj_proof3} for an example.

\begin{enumerate}[label=\arabic*)]
\item For every maximal run of consecutive $E$ steps in the descending portion of $P$ (to the right of $x=j$), which corresponds to entries
$e_{i}=e_{i+1}=\dots =e_{i+t-2}$ with $i>j$ and $t\geq 2$, turn the $N$ step in the ascending portion of $P$ going from height $e_i$ to height $e_i+1$ into an $N^*_t$ step.
\item Erase the descending portion of $P$, and instead append $j-e_j$ $N$ steps . Let $\varphi(e)\in \mathcal{\widetilde{R}}_n$ be the resulting path from the origin to $(j,j)$.
\end{enumerate}

It is clear that $\varphi$ is a bijection, so $|\bI_{n}\left(>,<,-\right)|=|\mathcal{\widetilde{R}}_{n}|$. In the rest of the proof we
count multi-marked Dyck paths with an unmarked tail, mimicking the proof of Theorem~\ref{thm:A071356}.

Define an elbow of a path in $\mathcal{\widetilde{P}}_{n}$ to be a consecutive occurrence of $EN$ or $EN^*_t$ for some $t\geq 2$. If $e\in\bI_{n}\left(>,<,-\right)$ and $R=\varphi(e)$, then $\dist(e)$ equals the following statistic on $R$, which we also denote by $\dist$:
\[
\dist(R) = \#\left\{\text{elbows in }R\right\} + \sum_{t\geq 2}\#\left\{N^*_t \text{ steps in }R\text{ that are not part of an elbow}\right\}.
\]

Now we find expressions for the generating functions
\[
\widetilde{R}(z,t)=\sum_{n\geq 0}\sum_{R\in\mathcal{\widetilde{R}}_{n}}z^{n}t^{\dist(R)}=I_{\left(>,<,-\right)}(z,t)
\quad\text{and}\quad
\widetilde{P}(z,t)=\sum_{n\geq 0}\sum_{P\in\mathcal{\widetilde{P}}_{n}}z^{n}t^{\dist(P)}.
\]
An identical argument to the one we used to deduce Equation~\eqref{eq:A071356_3} yields
\begin{equation}\label{eq:unimodal_eq_R}
\widetilde{R}(z,t) = \frac{1-z(1-t)\widetilde{P}(z,t)}{1-z\widetilde{P}(z,t)},
\end{equation}
so it is enough to compute $\widetilde{P}(z,t)$.

If $P$ is an irreducible multi-marked Dyck path, then there is a unique multi-marked Dyck path $P'$ such that either $P=EP'N$ or $P=EP'N^*_t$ for some $t\geq 2$. Paths $P$ for which $P'$ is empty contribute $zt+\frac{z^{2}t}{1-z}$ to the generating function $\widetilde{P}(z,t)$, because in this case $P$ is an elbow. On the other hand, paths $P$ for which $P'$ is nonempty contribute $\left(z+\frac{z^{2}t}{1-z}\right)\left(\widetilde{P}(z,t)-1\right)$ to $\widetilde{P}(z,t)$. Given that every path in $\mathcal{\widetilde{P}}_{n}$ has a unique decomposition as a sequence of irreducible multi-marked Dyck paths, it follows that
\[
\widetilde{P}(z,t) = \frac{1}{1-\left[z(t-1)+\left(z+\frac{z^{2}t}{1-z}\right)\widetilde{P}(z,t)\right]}.
\]
Solving for $\widetilde{P}(z,t)$ yields
\[
\widetilde{P}(z,t)=\frac{(1-z)\left(1+z-zt- \sqrt{(1+z-zt)^{2}-4\left(z\frac{z^{2}t}{1-z}\right)}\right)}{2z(1-z+zt)}.
\]
Hence, Equation~\eqref{eq:unimodal_eq_R} implies that
\[
\widetilde{R}(z,t) = \frac{1+z(2-t)-z^{2}(1-t)+(z-1)\sqrt{(1+z-zt)^{2}+\frac{4z(1-z+zt)}{z-1}}}{2z(2-z)}.\qedhere
\]
\end{proof}

By setting $t=1$ in Theorem~\ref{thm:unimodal_1}, we find that
\begin{equation}\label{eq:unimodal_0}
\sum_{n\geq 0}\left|\bI_{n}\left(>,<,-\right)\right|z^{n}=\frac{1+z-\sqrt{1-6z+5z^2}}{2z(2-z)},
\end{equation}
which appears as A033321 in~\cite{OEIS}. This is also the generating function for $\left|S_{n}\left(2143,3142,4132\right)\right|$, that is, permutations simultaneously avoiding the classical patterns $2143$, $3142$ and $4132$,
see~\cite[Lemma~3.2]{BloomBurstein}. Burstein and Stromquist~\cite{BursteinStromquist} point out that $\pi\rightarrow\Theta\left(\pi^{RC}\right)$ is a bijection between $S_{n}\left(2143,3142,4132\right)$ and $\bI_{n}\left(>,<,-\right)$.

\begin{figure}[htp]
\begin{centering}
\includegraphics[scale=0.55]{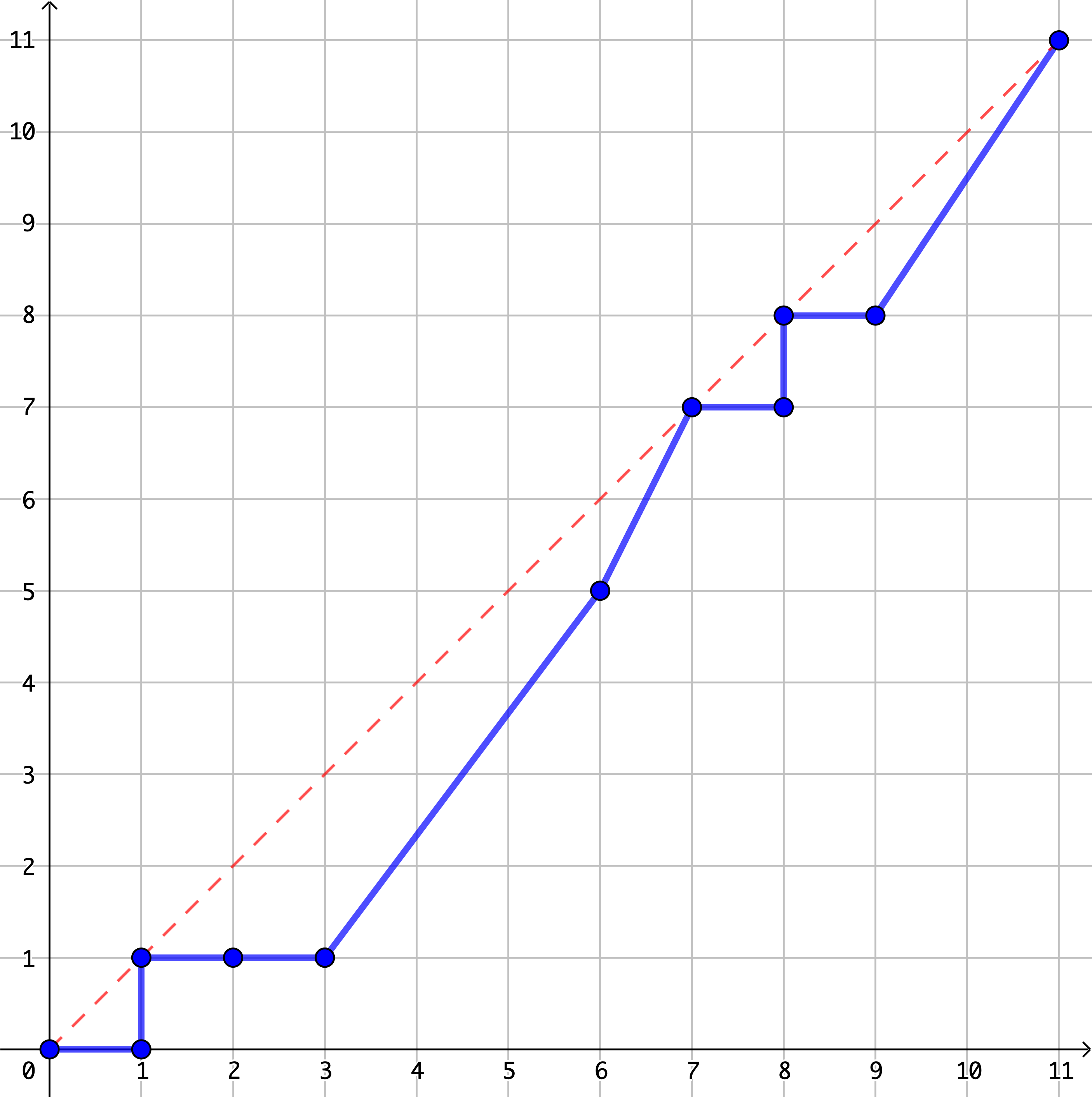}
\end{centering}
\caption{The path $R'=ENEED_{4}D_{2}ENED_{3}\in\mathcal{\widetilde{R}}'_{11}$ corresponding to $e=011345442111$
(from Figure~\ref{fig:carla_conj_proof3}) under the bijection $\bI_{12}(>,<,-)\to\mathcal{\widetilde{R}}'_{11}$.}\label{fig:gen_deutsch_path}
\end{figure}

In Section~\ref{subsec:A071356} we provided a bijection between $\bI_{n}\left(>,\leq,-\right)$ and $\mathcal{R}'_{n-1}$. To extend this construction to inversion sequences avoiding $\left(>,<,-\right)$, let $\mathcal{\widetilde{R}}'_{n}$ be the set of underdiagonal paths from $(0,0)$ to the line $x=n$ with steps $N$, $E$ and $D_{t}=(t-1,t)$ for $t\geq 2$, and note that $\mathcal{R}'_{n}\subseteq\mathcal{\widetilde{R}}'_{n}$.  We construct a bijection between $\bI_{n}\left(>,<,-\right)$ and $\mathcal{\widetilde{R}}'_{n-1}$ as follows. Given $e\in\bI_{n}\left(>,<,-\right)$, first let $R\in\mathcal{\widetilde{R}}_{n}$ be its corresponding multi-marked Dyck path with an unmarked tail. Then replace each marked step $N^{*}_{t}$ in $R$ by a step $D_{t}$, and delete the last run of $N$ steps as well as the $E$ step immediately preceding it, to obtain a path $R'\in\mathcal{\widetilde{R}}'_{n-1}$ (see Figure~\ref{fig:gen_deutsch_path} for an example). It follows from this bijection and Equation~\eqref{eq:unimodal_0} that
\begin{equation*}
\sum_{n\geq 0}\left|\mathcal{\widetilde{R}}'_{n}\right|z^{n}=\frac{1-3z+2z^{2}-\sqrt{1-6z+5z^2}}{2z^{2}(2-z)}.
\end{equation*}

Let $\left(R_{1},R_{2},R_{3}\right)$ be one of the patterns listed in Table~\ref{tab:unimodal}. Then $\bI_{n}\left(R_{1},R_{2},R_{3}\right)\subseteq\bI_{n}\left(>,<,-\right)$. Indeed, Table~\ref{tab:unimodal} shows that the unimodality condition characterizing $\bI_{n}\left(R_{1},R_{2},R_{3}\right)$ is weaker than that characterizing $\bI_{n}\left(>,<,-\right)$. Hence, an inversion sequence $e\in \bI_{n}\left(R_{1},R_{2},R_{3}\right)$ can still be represented as a multi-marked Dyck path $\varphi(e)$ with an unmarked tail, where $\varphi$ is the bijection in the proof of Theorem~\ref{thm:unimodal_1}. For each pattern in Table~\ref{tab:unimodal}, $\varphi$ induces a bijection between $\bI_{n}\left(R_{1},R_{2},R_{3}\right)$ and a subset of $\mathcal{\widetilde{R}}_{n}$ that can be characterized, allowing us to find the generating function $I_{\left(R_{1},R_{2},R_{3}\right)}(z,t)$
using the symbolic method, as in the proof of Theorem~\ref{thm:unimodal_1}.

\begin{thm}\label{thm:unimodal_2}
For patterns $\left(R_{1},R_{2},R_{3}\right)$ such that $\bI_{n}\left(R_{1},R_{2},R_{3}\right)$ is characterized by unimodality conditions, the generating functions $I_{\left(R_{1},R_{2},R_{3}\right)}(z,t)$ are as follows:

\begin{equation*}
I_{\left(<,-,<\right)}(z,t) = \frac{1-3z+zt+3z^{2}-2z^{2}t+z^{2}t^{2}-z^{3}+z^{3}t}{(1-z)^{3}},
\end{equation*}
\begin{equation*}
I_{\left(\neq,<,-\right)}(z,t) = \frac{1-4z+zt+6z^{2}-4z^{2}t+z^{2}t^{2}-4z^{3}+5z^{3}t-z^{3}t^{2}+z^{4}-2z^{4}t+z^{4}t^{2}}{(1-z)^{2}(1-2z+z^{2}-z^{2}t)},
\end{equation*}
\begin{equation*}
I_{\left(\neq,\leq,-\right)}(z,t) = \frac{1-2z+zt+z^{2}-2z^{2}t+z^{2}t^{2}+z^{3}t}{(1-z)(1-z-z^{2}t)},
\end{equation*}
\begin{equation}\label{eq:unimodal_4}
I_{\left(>,<,-\right)}(z,t) = \frac{1+z(2-t)-z^{2}(1-t)-(1-z)\sqrt{(1+z-zt)^{2}-\frac{4z(1-z+zt)}{1-z}}}{2z(2-z)},
\end{equation}
\begin{equation}\label{eq:unimodal_5}
I_{\left(>,\leq,-\right)}(z,t) = \frac{1+z(3-t)-\sqrt{1-z\left(2+2t-z+6zt-zt^{2}\right)}}{4z},
\end{equation}
\begin{equation*}
I_{\left(>,\neq,-\right)}(z,t) = \frac{1-2z+z^{2}(1-t)^{2}+(1-z+zt)\sqrt{(1+z-zt)^{2}-4z}}{2(1-z)\sqrt{(1+z-zt)^{2}-4z}},
\end{equation*}
\begin{align*}
&I_{\left(\geq,\neq,-\right)}(z,t) = \frac{1-z-zt+2z^{2}t}{(1-z)(1-zt)^{2}}, \quad &&
I_{\left(=,<,-\right)}(z,t) = \frac{1-z}{1-z-zt},\\
&I_{\left(=,\leq,-\right)}(z,t) = \frac{1}{1-zt-z^{2}t}, \quad &&
I_{\left(\geq,\leq,\neq\right)}(z,t) = \frac{1-z+z^{3}t}{(1-z)(1-zt-z^{2}t)}.
\end{align*}
\end{thm}

Equations~\eqref{eq:unimodal_4} and~\eqref{eq:unimodal_5} are restatements of Theorems~\ref{thm:unimodal_1} and~\ref{thm:A071356}, respectively, included here for the sake of completeness.
By setting $t=1$ in each of the generating functions $I_{\left(R_{1},R_{2},R_{3}\right)}(z,t)$ in Theorem~\ref{thm:unimodal_2}, we obtain expressions for $\sum_{n\geq 0}\left|\bI_{n}\left(R_{1},R_{2},R_{3}\right)\right|z^{n}$ in each case, from where
one can recover the results listed in Table~\ref{tab:unimodal} with some algebraic manipulations.

It is natural to ask if an analogue of Corollary~\ref{cor:A071356} holds for any of the patterns in Table~\ref{tab:unimodal}.
In addition to $\left(>,\leq,-\right)$, the only other pattern on this list for which the distribution of $\dist$ on $\bI_{n}\left(R_{1},R_{2},R_{3}\right)$ is symmetric is $\left(=,<,-\right)$.
In this case,
$I_{\left(=,<,-\right)}(z,t)=1+\sum_{n\ge 1} t(t+1)^{n-1}  z^n$.
Computations for small values of $n$ also suggest that the distribution of $\dist$ on $\bI_{n}\left(R_{1},R_{2},R_{3}\right)$ is unimodal for each of the patterns in Table~\ref{tab:unimodal}.

\end{document}